\documentclass[11pt]{article}
\usepackage{amssymb}
\usepackage{mathrsfs}
\usepackage{latexsym,amsmath,amssymb,amsfonts,epsfig,graphicx,cite,psfrag}
\usepackage{eepic,color,colordvi,amscd}
\usepackage{ebezier}
\usepackage{verbatim}
\usepackage{subfigure}

\usepackage{amsthm}
\usepackage{comment}

\theoremstyle{plain}
\newtheorem{theo}{Theorem}[section]

\newtheorem{lem}[theo]{Lemma}

\theoremstyle{definition}

\theoremstyle{remark}

\def\qed{\hfill \rule{4pt}{7pt}}

\begin{document}
\title{4-Separations in Haj\'{o}s graphs\footnote{Partially supported by NSF grant DMS-1600738.
Email: qqxie@fudan.edu.cn (Q. Xie), shijie.xie@gatech.edu (S. Xie), yu@math.gatech.edu (X. Yu), and  xyuan@gatech.edu (X. Yuan)}}

\author{Qiqin Xie\\
Shanghai Center for Mathematical Sciences\\
Fudan University\\
Shanghai, China 200438\\
\medskip\\
 Shijie Xie, Xingxing Yu, and Xiaofan Yuan\\
School of Mathematics\\
Georgia Institute of Technology\\
Atlanta, GA 30332--0160, USA }

\date{April 26, 2020}

\maketitle

\begin{abstract}

As a natural extension of the Four Color Theorem, Haj\'{o}s conjectured that graphs containing no $K_5$-subdivision are  4-colorable. Any possible counterexample to this conjecture with minimum
number of vertices is called a {\it  Haj\'{o}s graph}. Previous results
show that Haj\'{o}s graphs are 4-connected but not 5-connected. A $k$-separation in a graph $G$ is a pair $(G_1,G_2)$ of edge-disjoint subgraphs of $G$ such that
$|V(G_1\cap G_2)|=k$, $G=G_1\cup G_2$, and $G_i\not\subseteq G_{3-i}$ for $i=1,2$. In this paper, we show that   Haj\'{o}s graphs
do not admit a 4-separation $(G_1,G_2)$ such that  $|V(G_1)|\ge 6$ and $G_1$ can be drawn in the plane with no edge crossings and all vertices in $V(G_1\cap G_2)$
incident with a common face. This is a step in our attempt to reduce Haj\'{o}s' conjecture to the Four Color Theorem.

\bigskip

AMS Subject Classification: 05C10, 05C40, 05C83

Keywords: Wheels, coloring, graph subdivision, disjoint paths
\end{abstract}


\newpage
\section{Introduction}

Using  Kuratowski's characterization of planar graphs   \cite{Ku30},
the Four Color Theorem \cite{Ap77a, Ap77b, Ap89, FCT} can be stated as follows: Graphs containing no $K_5$-subdivision or $K_{3,3}$-subdivision are 4-colorable.
Since $K_{3,3}$ has chromatic number 2, it is natural to expect that graphs containing no $K_5$-subdivision are also 4-colorable. Indeed,
this is part of a more general conjecture made by
Haj\'{o}s in the 1950s (see \cite{Th05}, although reference
\cite{Ha61} is often cited): For any positive integer $k$, every graph
not containing $K_{k+1}$-subdivision is $k$-colorable. It is not hard to
prove this conjecture  for $k \le 3$. However,
Catlin \cite{Ca79} disproved Haj\'{o}s' conjecture for $k \ge
6$.  Erd\H{o}s and Fajtlowicz \cite{EF81} then showed that
Haj\'{o}s' conjecture fails for almost all graphs. On the other hand,
K\"{u}hn and Osthus \cite{KO02} proved that Haj\'{o}s' conjecture
holds for graphs with large girth, and Thomassen \cite{Th05} pointed
out interesting connections between Haj\'{o}s' conjecture and several
important problems, including Ramsey numbers, Max-Cut, and perfect
graphs.  Haj\'{o}s' conjecture remains open for $k=4$ and $k=5$.

In this paper, we are concerned with Haj\'{o}s' conjecture for
$k=4$. We say that a graph $G$ is a {\it Haj\'{o}s graph} if
\begin{itemize}
\item [(1)] $G$ contains no $K_5$-subdivision,
\item [(2)] $G$ is not 4-colorable, and
\item [(3)] subject to (1) and (2), $|V(G)|$ is minimum.
\end{itemize}
Thus, if no Haj\'{o}s graph exists then graphs not containing
$K_5$-subdivisions are 4-colorable.

Recently, He, Wang, and Yu \cite{KSC1, KSC2,
  KSC3, KSC4} proved that every 5-connected nonplanar graph
contains a $K_5$-subdivision, establishing a conjecture of Kelmans
\cite{Ke79} and, independently, of Seymour \cite{Se77} (also see Mader
\cite{Ma98}).  Therefore, Haj\'{o}s graphs cannot be 5-connected. On the other hand,
Yu and Zickfeld \cite{Yu06} proved that  Haj\'{o}s graphs must be
4-connected, and Sun and Yu \cite{Su16} proved that for any 4-cut $T$
in an Haj\'{o}s graph $G$, $G-T$ has exactly 2 components.

The goal of this paper is to prove a result useful for modifying  the recent proof of the Kelmans-Seymour conjecture  in  \cite{KSC1, KSC2, KSC3,KSC4} to
make progress on the Haj\'{o}s conjecture; in particular, for the class of graphs containing
$K_4^-$ as a subgraph, where $K_4^-$ is the graph obtained from $K_4$ by removing an edge.

To state our result precisely, we need some notation. Let $G_1$, $G_2$
be two graphs.
We use $G_1 \cup G_2$ (respectively, $G_1\cap G_2$) to denote the graph with vertex set
$V(G_1) \cup V(G_2)$ (respectively, $V(G_1)\cap V(G_2)$) and edge set
$E(G_1) \cup E(G_2)$ (respectively, $E(G_1)\cap E(G_2)$).   Let $G$ be a
graph and $k$ a nonnegative integer; then a \textit{k-separation} in $G$ is a pair $(G_1,G_2)$ of edge-disjoint
subgraphs $G_1,G_2$ of $G$ such that $G = G_1 \cup G_2$,  $|V(G_1\cap G_2)|=k$, and $G_i\not\subseteq G_{3-i}$ for $i=1,2$.

Let $G$ be a graph and $S\subseteq V(G)$.
For convenience, we say that $(G,S)$ is {\it planar} if $G$ has a drawing in a closed disc in the plane with no edge crossings and
with vertices in $S$ on the boundary of the disc. We often assume that we work with such an embedding when we say $(G,S)$ is planar.
Two elements of $V(G)\cup E(G)$ are said to be {\it cofacial} if they are incident with a common face.
Our main result can be stated as follows, it will be used in subsequent work to derive further useful structure of Haj\'{o}s graphs.

\begin{theo}\label{main}
If $G$ is a Haj\'os graph and $G$ has a 4-separation $(G_1, G_2)$ such that $(G_1,
V(G_1\cap G_2))$ is planar then $|V(G_1)| \le 5$.
\end{theo}

To prove Theorem~\ref{main}, we first find a special wheel inside $G_1$, then extend the wheel to
$V(G_1\cap G_2)$  by four disjoint paths inside $G_1$, and  form a $K_5$-subdivision with
two disjoint paths  in $G_2$.
By a  \textit{wheel} we mean a graph which consists of a cycle $C$,  a vertex $v$ not on $C$ (known as the {\it center} of the wheel),
and at least three edges from $v$ to a subset of $V(C)$.  The wheels in this paper are special -- they are inside a plane graph consisting of vertices and edges
that are cofacial with a given vertex. For any positive integer $k$,
let $[k]:=\{1, \ldots, k\}$.

Let $G$ be a graph and $(G',G'')$ be a separation in $G$ such that $G'$ is drawn in a closed disc in the plane
with no edge crossings and $V(G'\cap G'')$ is contained in the boundary of that disc.
Let $w\in V(G')\setminus V(G'')$ such that the vertices and edges of $G'$ cofacial with $w$ form a wheel, denoted as $W$.
(Note such $W$ is well defined when $G$ is 3-connected and $w$ is not incident with the unbounded face of $G'$.) Thus $N_G(w)\subseteq V(W)$.
We say that
$W$ is {\it $V(G'\cap G'')$-good} if $V(G'\cap G'')\cap V(W) \subseteq N_G(w)$.
For any $S\subseteq V(G'\cap G'')$ with $|S|\le 4$, we say that $W$ is {\it $(V(G'\cap G''),S)$-extendable} if $G$ has four paths $P_1,P_2,P_3,P_4$ from $w$ to $V(G'\cap G'')$ such that
      \begin{itemize}
             \item  $V(P_i\cap P_j)=\{w\}$ for all distinct $i,j\in [4]$,
          \item  $|V(P_i-w)\cap V(W)|=1$ for $i\in [4]$, and
            \item for any $s\in S$ there exists $i\in [4]$ such that $P_i$ is from $w$ to $s$.
      \end{itemize}
Note that each $P_i$ may use more than one vertex from $V(G'\cap G'')$. When $S=V(G'\cap G'')$ we simply say that $W$ is
{\it $V(G'\cap G'')$-extendable}.

For the proof of Theorem~\ref{main}, we suppose $G$ has a 4-separation  $(G_1, G_2)$ such that $(G_1, V(G_1\cap G_2))$ is planar and $|V(G_1)| \ge 6$.
A result from \cite{XXYY19} shows that $G_1$ has a $V(G_1\cap G_2)$-good wheel. However, we need to allow the separation
$(G_1,G_2)$ to be a 5-separation in order to deal with issues when such wheels are not $V(G_1\cap G_2)$-extendable.
Another result from \cite{XXYY19} characterize all such 5-separations $(G_1,G_2)$ with $G_1$ containing no $V(G_1\cap G_2)$-good
wheel.  In Sections 2 and 3, we characterize the situations where good wheels are also extendable.
We complete the proof of Theorem~\ref{main} in Section 4.

It will be convenient to use a sequence of vertices to represent a path or cycle, with consecutive vertices representing an edge in the path.
Let $G$ be a graph. For $v\in V(G)$, we use $N_G(v)$ to denote the neighborhood of $v$ in  $G$. Let $T \subseteq V(G)$.
We use $G-T$ to denote the subgraph of $G$ induced by $V(G) \setminus  T$ and write $G-x$ when $T=\{x\}$. For any set $S$ of 2-element subsets of $V(G)$, we use
$G+S$ to denote the graph with $V(G+S)=V(G)$ and $E(G+S)=E(G)\cup S$, and write $G+xy$ if $S=\{\{x,y\}\}$.

Let $C$ be a cycle in a plane graph, and let  $u,v\in V(C)$. If $u=v$  let $uCv=u$, and if $u\ne v$ let
$uCv$ denote the subpath of $C$ from $u$ to $v$ in clockwise order.


\section{Extending a wheel}

In \cite{Yu06} it is shown that Haj\'{o}s graphs are 4-connected, and in \cite{KSC4} it is shown that Haj\'{o}s graphs are not 5-connected. So we have the following result.

\begin{lem}\label{4conn}
Haj\'{o}s graphs are 4-connected but not 5-connected.
\end{lem}

We also need a result from \cite{XXYY19} which characterizes the 4-separations and  5-separations $(G_1,G_2)$ with $(G_1,V(G_1\cap G_2))$ planar such that $G_1$ has no
$V(G_1\cap G_2)$-good wheel.
See Figure~\ref{obstructions} below for the graph $G_1$, where the solid vertices are in $V(G_1)\setminus V(G_2)$.
\begin{figure}[h]
        \includegraphics[width=\textwidth]{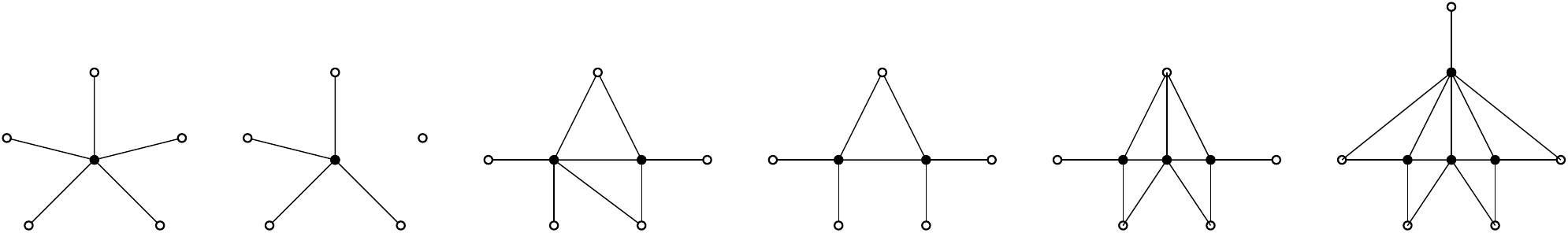}
           \caption{Obstructions to good wheels inside 5-separations.}
             \label{obstructions}
  \centering
    \end{figure}

\begin{lem}\label{5cutobs}
Let $G$ be a Haj\'os graph and  $(G_1, G_2)$ be a separation in $G$ such that $4\le |V(G_1\cap G_2)|\le 5$, $V(G_1\cap G_2)$ is independent in $G_1$, $(G_1,V(G_1\cap G_2))$ is planar, and
$V(G_1)\setminus V(G_2)\ne \emptyset$.
Then,  one of the following holds:
\begin{itemize}
\item [$(i)$] $G_1$ contains a $V(G_1\cap G_2)$-good wheel.
\item [$(ii)$] $|V(G_1\cap G_2)|=4$ and $|V(G_1)|=5$.
\item [$(iii)$] $|V(G_1\cap G_2)|=5$, $G_1$ is one of the graphs in Figure~\ref{obstructions} with $V(G_1)\setminus V(G_2)$ consisting of the solid vertices,
and if $|V(G_1)|=8$ then the degree 3 vertex in $G_1$  has degree at least 5 in $G$.
\end{itemize}
\end{lem}

Let $G$ be a Haj\'os graph and  $(G_1, G_2)$ be a separation in $G$ such that $|V(G_1\cap G_2)|\in \{4,5\}$ and $(G_1, V(G_1\cap G_2))$ is planar, and assume $W$ is a $V(G_1\cap G_2)$-good wheel in $G_1$.
      We wish to extend  $W$ to a $K_5$-subdivision by adding two disjoint paths, which
      must be routed through the non-planar part $G_2$.
        The following lemma
       provides four paths extending one such good wheel to $V(G_1\cap G_2)$.

    \begin{lem} \label{extension-5cut}
    	Let $G$ be a Haj\'os graph and let $(G_1,G_2)$ be a separation in $G$ with $V(G_1\cap G_2)$ independent in $G_1$ such that
     \begin{itemize}
      \item [$(i)$]  $|V(G_1\cap G_2)|\le 5$, $(G_1, V(G_1\cap G_2))$ is
       planar, and $G_1$ has a $V(G_1\cap G_2)$-good wheel,
      \item [$(ii)$] subject to $(i)$, $G_1$ is minimal, and
      \item [$(iii)$] subject to $(ii)$, $V(G_1\cap G_2)$ is minimal.
      \end{itemize}
 Then any $V(G_1\cap G_2)$-good wheel in $G_1$ is $V(G_1\cap G_2)$-extendable.
\end{lem}

\begin{proof} By our convention,  $G_1$ is drawn in a closed disc in the plane with no edge crossing such that
      $V(G_1\cap G_2)$ is on the boundary of that disc. Let $W$ be a $V(G_1\cap G_2)$-good wheel in $G_1$ with center $w$,
      $U=V(W-w)\setminus N_G(w)$, and  $G_1'=G_1-U$. If $G_1'$ has four disjoint paths from $N_G(w)$ to $V(G_1\cap G_2)$ then extending these paths to
      $w$ (by adding one edge for each path), we see that $W$ is  $V(G_1\cap G_2)$-extendable.
        So we may assume that such four paths do not exist. Then $G_1'$ has a separation
      $(H_1,H_2)$ such that $|V(H_1\cap H_2)|\le 3$, $N_G(w)\cup \{w\}\subseteq V(H_1)$, and $V(G_1\cap G_2)\subseteq V(H_2)$. We choose $(H_1,H_2)$ with $|V(H_1\cap H_2)|$ minimum.

We see that $V(H_1\cap H_2)\cup U$ is a cut in $G_1$ separating $H_1$ from $V(G_1\cap G_2)$.
Thus, by the planarity of $(G_1,V(G_1\cap G_2))$, we can draw  a simple closed curve $\gamma$ in the
      plane such that $\gamma \cap G_1\subseteq V(H_1\cap H_2)\cup U$,  $H_1$
      is inside $\gamma$, and $H_2$ is outside $\gamma$.  We choose such
      $\gamma$ that $|\gamma \cap G_1|$ is minimum.

Note  $V(H_1\cap H_2)\subseteq \gamma$ by the minimality of $|V(H_1\cap H_2)|$. Moreover, $\gamma\cap U\ne \emptyset$ as, otherwise, $V(H_1\cap H_2)$ would be a cut in $G$,
a contradiction as $G$ is 4-connected.

      For convenience, let $N_G(w)=\{w_1,\ldots, w_t\}$, and assume that the notation is chosen so that $w_1, \ldots, w_t$ occur on $W-w$ in clockwise order.
Moreover, for $i\in [t]$, let $W_i$ denote the path in $W-w$ from $w_i$ to $w_{i+1}$ in clockwise order, where $w_{t+1}=w_1$.
     We claim that
      \begin{itemize}
         \item [(1)]  any two vertices of $\gamma \cap U$ consecutive on $\gamma$
     must be contained in the same  $W_i$, for some $i\in [t]$.
      \end{itemize}
For, otherwise, let $u,v\in \gamma \cap U$ be consecutive on $\gamma$ such that $u\in V(W_i)$ and
     $v\in V(W_j)$, with $i<j$. Then we see that $G$ has a separation
     $(L_1,L_2)$ such that $V(L_1\cap L_2)=\{u,v,w\}$, $w_{i+1}\in
     V(L_1-L_2)$, and $G_2\subseteq L_2$. This contradicts the fact that
     $G$ is 4-connected. $\Box$

\medskip

     Note that $k\ge 2$. For, otherwise, it follows from (1) that $\gamma \cap U\subseteq W_i$ for some $i\in [k]$. Choose $u,v\in \gamma\cap U$ with $uW_iv$ maximal.
    Then $\{u,v\}\cup V(H_1\cap H_2)$ is a cut in $G$, a contradiction.

 Let $V(H_1\cap H_2)=\{v_1, \ldots, v_k\}$, where $2\le k\leq 3$, and for $i\in [k]$, let $\gamma_i$ be the curve in $\gamma$
   from $v_i$ to $v_{i+1}$ in clockwise order, where $v_{k+1}=v_1$.
We further claim that

    \begin{itemize}
    \item [(2)] there exist unique $i\in [k]$ and unique
     $j\in [t]$, for which $\gamma_i\cap W_j\ne \emptyset$.
    \end{itemize}
    For, suppose otherwise.
First,  assume that there exist $\gamma_i$ and $\gamma_l$ with $i\ne l$ such that for some $W_j$, $\gamma_i\cap W_j\ne \emptyset$
 and $\gamma_l\cap W_j\ne \emptyset$. Without loss of generality, we may assume $i=1$ and $l=2$. Then, by planarity and by (1),
$U\cup (V(H_1\cap H_2)\setminus \{v_2\})$ is a cut in $G_1$ separating $\{w\}\cup N_G(w)$ from $V(G_1\cap G_2)$; so $G_1'$ has a separation $(H_1',H_2')$ such that $V(H_1'\cap H_2')=V(H_1\cap H_2)\setminus \{v_2\}$,
$N_G(w)\cup \{w\}\subseteq V(H_1')$, and $V(G_1\cap G_2)\subseteq V(H_2')$. This contradicts the choice of $(H_1,H_2)$ that $|V(H_1\cap H_2)|$ is minimum.

Hence, by (1),  there exist $p\neq q$ and $i\neq j$ such that $\gamma_p\cap W_i\ne \emptyset$
    and $\gamma_q\cap W_j\ne \emptyset$. Without loss of generality, we may further assume that $p=1, q=2$, and $i<j$.
       Let $v_2'\in V(W_i)$ such that
     $v_2,v_2'$ are consecutive on  $\gamma_1$, and $v_2''\in V(W_j)$ such that $v_2,v_2''$ are
     consecutive on $\gamma_2$. Then, by (1), $G$ has a 4-separation
     $(L_1,L_2)$ such that $V(L_1\cap L_2)=\{v_2, v_2',v_2'',w\}$ is independent in $L_1$,
     $\{w_{i+1},\ldots, w_j\}\subseteq V(L_1)$, and $G_2\subseteq L_2$.  If $|V(L_1)|\ge 6$ then, by Lemma~\ref{5cutobs}, $L_1$ has a
     $V(L_1\cap L_2)$-good wheel; so $(L_1,L_2)$ contradicts the choice of $(G_1,G_2)$. Hence,  $|V(L_1)|= 5$ and $j=i+1$.

We may assume $k=3$. For, suppose $k=2$. Let $v'\in V(W_i)$ such that $v_1, v'$ are consecutive on $\gamma_1$, and let $v''\in V(W_{i+1})$ such that $v_1, v''$ are consecutive on $\gamma_2$. By (1),  $G$ has a 4-separation
$(L_1',L_2')$ such that $V(L_1'\cap L_2')=\{v_1, v',v'',w_{i+1}\}$ is independent in $L_1'$,
$\{w\}\cup ( N_G(w)\setminus \{w_{i+1}\}) \subseteq V(L_1')$, and $G_2\subseteq L_2'$.  Since $|V(L_1')|\ge 6$, it follows from
Lemma~\ref{5cutobs} that $L_1'$ contains a $V(L_1'\cap L_2')$-good wheel. So $(L_1',L_2')$ contradicts the choice of $(G_1, G_2)$.

Now let $v_1'\in V(W_i)$ such that $v_1, v_1'$ are consecutive on $\gamma_1$, and let $v_3'\in V(W_{i+1})$ such that $v_3, v_3'$ are consecutive on $\gamma_2$.

Suppose $\gamma_3\cap U=\emptyset$.  Then $v_1\ne w_1$ or $v_3\ne w_3$;
otherwise,  $\{v_1,v_3,w_{i+1}\}$ would be a 3-cut in $G$. If $v_1=w_1$ then by (1), $G$ has a separation $(L_1', L_2')$ such that $V(L_1'\cap L_2')=\{v_1, v_3, v_3', w_{i+1}\}$ is independent in $L_1'$,
$\{w\}\cup (N_G(w)\setminus \{w_{i+1}\})\subseteq V(L_1')$, and $G_2\subseteq L_2'$; so by  Lemma~\ref{5cutobs}, $L_1'$ has a $V(L_1'\cap L_2')$-good wheel and, hence, $(L_1',L_2')$ contradicts the choice of $(G_1,G_2)$.
So $v_1\ne w_1$. Similarly, $v_3\ne w_3$.
Then by (1), $G$ has a separation $(L_1', L_2')$ such that $V(L_1'\cap L_2')=\{v_1, v_1', v_3, v_3', w_{i+1}\}$ is independent in $L_1'$,
$\{w\}\cup (N_G(w)\setminus \{w_{i+1}\})\subseteq V(L_1')$,  $G_2\subseteq L_2'$, and $|V(L_1')\setminus V(L_2')|\ge 4$. By the choice of $(G_1,G_2)$,
$L_1'$ does not admit a $V(L_1'\cap L_2')$-good wheel. So
by Lemma~\ref{5cutobs}, $|V(L_1'-L_2')|=4$ and
$(L_1', V(L_1'\cap L_2'))$ is the 9-vertex graph in Figure~\ref{obstructions}, which means that
the only neighbor of $w_{i+1}$ in $L_1'$, namely $w$, should have degree 6 in $L_1'$ and must be adjacent to $v_1'$ and $v_3'$. But this is a contradiction as $v_1',v_3'\notin N_G(w)$.

So $\gamma_3\cap U\neq\emptyset$. But then by (1) and 4-connectedness of $G$, there exist $l\in \{1, 3\}$ and vertex $v_l''$ such that $v_l',v_l,v_l''$
are consecutive on $\gamma$ in order listed
and $G$ has a 4-separation $(L_1',L_2')$ with $V(L_1'\cap L_2')=\{w,v_l',v_l,v_l''\}$ independent in
$L_1'$, $G_2\subseteq L_2'$, and $|V(L_1')|\ge 6$.  Then by Lemma~\ref{5cutobs}, $L_1'$ contains a $V(L_1'\cap L_2')$-good wheel. Hence $(L_1',L_2')$ contradicts the choice of $(G_1, G_2)$.
$\Box$

\medskip

     Thus, by (1) and (2), we may assume that $\gamma_1\cap W_1\ne \emptyset$ and, for $i\in [k]\setminus \{1\}$ and $j\in [t]\setminus \{1\}$,
      $\gamma_i\cap W_j=\emptyset$. Let $v_1',v_2'\in V(W_1)$ such that, for $i=1,2$,  $v_i$ and $v_i'$ are consecutive on $\gamma_1$.
    Then $G$ has a  separation $(L_1,L_2)$ such that $V(L_1\cap L_2)=\{v_i : i\in
     [k]\}\cup \{v_1',v_2'\}$ is independent in $L_1$, $N_G(w)\cup \{w\}\subseteq V(L_1)$,
     and $G_2\subseteq L_2$.  Note that $w\in V(L_1)\setminus V(L_2)$.   Also note that  $v_1\ne w_1$ or $v_2\ne w_2$; otherwise, $V(L_1\cap L_2)$ would be a 3-cut in $G$.
    If $v_1=w_1$ or $v_2=w_2$ then $|V(L_1\cap L_2)|=4$ and $|V(L_1)\setminus V(L_2)|\ge 6$; hence, by Lemma~\ref{5cutobs}, $L_1$ has a $V(L_1\cap L_2)$-good wheel and, hence, $(L_1, L_2)$ contradicts the choice of $(G_1,G_2)$.
     So $v_1\ne w_1$ and $v_2\ne w_2$. Hence $|V(L_1)|\ge 9$ and
     $w$ is not adjacent to $\{v_1,v_2,v_1',v_2'\}$. It follows from Lemma~\ref{5cutobs} that $L_1$ contains a $V(L_1\cap L_2)$-good wheel.
     Hence $(L_1,L_2)$ contradicts the choice of $(G_1, G_2)$.
\end{proof}

   To extend a wheel to a $K_5$-subdivision, we need the following weaker version of a result of Seymour \cite{Se80}, with equivalent forms
  proved in \cite{CR79,Sh80, Th80}. For a graph $G$ and vertices $v_1, v_2,\ldots, v_n$  of $G$, we say that  $(G, v_1, v_2,\ldots, v_n)$ is planar
  if  $G$ can be drawn in a closed disc in the plane with no edge crossings such that  $v_1, v_2, \cdots, v_n$ occur on the boundary of the disc
  in clockwise order.

   \begin{lem}\label{2link}
   Let $G$ be a graph and $s_1,s_2,t_1,t_2\in V(G)$ be distinct such that, for any $S\subseteq V(G)$ with $|S|\le 3$, every component of $G-S$
   must contain a vertex from $\{s_1,s_2,t_1,t_2\}$. Then either $G$ contains disjoint paths from $s_1,s_2$ to $t_1,t_2$, respectively, or
   $(G,s_1,s_2,t_1,t_2)$ is planar.
   \end{lem}

   The next result shows that in a Haj\'{o}s graph, we cannot extend
   a wheel in certain way.

  \begin{lem}\label{nogoodwheel}
  Let $G$ be a Haj\'os graph. Suppose there exists a $4$-separation $(G_1,G_2)$ in $G$ such that $(G_1, V(G_1\cap G_2))$ is
       planar. If $W$ is a $V(G_1\cap G_2)$-good wheel in $G_1$ then $W$ is not $V(G_1\cap G_2)$-extendable.
  \end{lem}
\begin{proof}

For, suppose $W$ is $V(G_1\cap G_2)$-extendable. Let $V(G_1\cap G_2)=\{t_1, t_2, t_3, t_4\}$, and assume that the notation is chosen so that
      $(G_1,t_1,t_2,t_3,t_4)$ is planar.  Then there exist four paths $P_1, P_2, P_3, P_4$  in $G_1$ from $w$ to $t_1, t_2, t_3, t_4$, respectively,
      such that  $V(P_i\cap P_j)=\{w\}$ for any distinct $i,j\in [4]$ and $|V(P_i\cap W)|=2$ for $i\in [4]$.

  If $(G_2,t_1,t_2,t_3,t_4)$ is planar then $G$ is planar and, hence, 4-colorable, a contradiction. So $(G_2,t_1,t_2,t_3,t_4)$ is not planar.
 Then, by Lemma~\ref{2link},  $G_2$ has disjoint paths $Q_1,Q_2$ from $t_1,t_2$ to $t_3,t_4$, respectively. But then
$W\cup P_1\cup P_2\cup P_3 \cup P_4\cup Q_1\cup Q_2$ is a $K_5$-subdivision in $G$, a contradiction.
 \end{proof}

\section{Extending paths from 5-cuts to 4-cuts}

  The goal of this section is to describe the situations where a
    good wheel cannot be extended from a 5-cut to a 4-cut in the desired way. We achieve this goal in three steps (formulated as lemmas), by gradually
reducing the number of possibilities.  The first lemma has four possibilities.

\begin{lem}\label{extension-4cut}
Suppose $G$ is a Haj\'{o}s graph and $(G_1,G_2)$ is a $4$-separation in $G$ such that
 $(G_1, V(G_1\cap G_2))$ is planar and  $|V(G_1)|\ge 6$, and, subject to this, $G_1$ is minimal.
 Moreover, suppose that $G_1$ has a 5-separation $(H,L)$ with  $V(G_1\cap G_2)\subseteq V(L)$ and $V(H\cap L)$ independent in $H$, such that
  \begin{itemize}
   \item [$(a)$] $V(G_1\cap G_2) \not\subseteq V(H\cap L)$ and
$H$ has a $V(H\cap L)$-good wheel,
   \item [$(b)$] subject to $(a)$, $|S|$ is minimum, where $S=V(H\cap
     L)\cap V(G_1\cap G_2)$, and
   \item  [$(c)$] subject $(b)$, $H$ is minimal.
  \end{itemize}
Then  $H$ has  a $(V(H\cap L),S)$-extendable wheel, or $G_1$ has a $V(G_1\cap G_2)$-extendable wheel,
or, for each $V(H\cap L)$-good wheel $W$ with center $w$ in $H$, one of the following holds:
             \begin{itemize}
               \item [$(i)$] There exist $s\in S\setminus V(W)$ and $a,
                 b\in V(W-w)\setminus N_G(w)$ such that
                 $N_H(s)=\{a,b\}$ and either $a=b$ or  $ab\in E(W)$.
               \item [$(ii)$] There exist $s_1,s_2\in S\setminus V(W)$,
                 $a, b\in V(W-w)\setminus N_G(w)$, and separation
                 $(H_1,H_2)$ in $H$ such that
                 $V(H_1\cap H_2)=\{a,b,w\}$, $\{s_1, s_2\}\subseteq V(H_1)$, $|N_G(w)\cap V(H_1)|=1$,
                 and $V(H\cap L)\setminus \{s_1, s_2\}\subseteq V(H_2)$.
           \item [$(iii)$] $|S|=3$, and
              there exist $s_1,s_2\in S$,
              $a,b\in V(W-w)\setminus N_G(w)$, and
               separation $(H_1,H_2)$ in $H$ such that
                        $V(H_1\cap H_2)=\{a,b,s_1,s_2\}$,
                          $S\subseteq
                          V(H_1)$, and $\{w\}\cup (V(H\cap L)\setminus
                         S)\subseteq V(H_2)$.
            \item [$(iv)$] There exist $a,b\in V(W-w)\setminus N_G(w)$, $c\in V(H)\setminus V(W)$,
                      and a separation $(H_1,H_2)$ in $H$ such that $V(H_1\cap H_2)=\{a,b,c\}$, $(N_G(w)\cup \{w\})\cap V(H_1)=\emptyset$,
                  $|V(H_1)\cap V(H\cap L)|=2$, $V(H_1)\cap V(H\cap L)\subseteq S$, and $N_G(w)\cup \{w\}\subseteq V(H_2)\setminus V(H_1)$.
             \end{itemize}
\end{lem}

  \begin{proof}
  Note that $|S|\le 3$ as $V(G_1\cap G_2)\not\subseteq V(H\cap L)$. We may assume that $G_1$ is drawn in a closed disc in the plane with no edge crossing
  such that $V(G_1\cap G_2)$ is on the
  boundary of that disc.
  For convenience, let $V(H\cap  L)=\{t_i : i\in [5]\}$ such that $(H, t_1,t_2,t_3,t_4, t_5)$ is planar.
    Let $D$ denote the outer walk of $H$.
Let $W$ be a $V(H\cap L)$-good wheel in $H$ with center $w$, and let $F=W-w$ (which is a cycle).

  By Lemma~\ref{extension-5cut}, $W$ is $V(H\cap L)$-extendable in $H$.
  Without loss of generality, assume that $H$ has four paths $P_1,P_2,P_3,P_4$ from $w$ to $t_1,t_2,t_3,t_4$,
  respectively, such that $|V(P_i\cap F)|=1$ and $t_5\notin V(P_i)$
  for $i\in [4]$.  Moreover, we may assume $t_5\in S$ as, otherwise, these paths show that $W$ is $(V(H\cap L),S)$-extendable. Then
  $t_5\notin V(W)$; for, if $t_5\in V(W)$ then $t_5w\in E(H)$ (as $W$
  is $V(H\cap L)$-good)  which, combined with three of $\{P_1,P_2,P_3,P_4\}$, shows that
    $W$ is $(V(H\cap L),S)$-extendable.
  Let $V(P_i\cap F)=\{w_i\}$ for $i\in [4]$. Since $(H, t_1,t_2,t_3,t_4, t_5)$ is planar,   $w_1,w_2,w_3,w_4$ occur on $F$ in clockwise order.

   We choose $P_1,P_4$ so that $w_4Fw_1$ is minimal. Then
      $$N_G(w)\cap V(w_4Fw_1-\{w_1,w_4\})=\emptyset.$$
   For, suppose not and let $w'\in N_G(w)\cap V(w_4Fw_1-\{w_1,w_4\})$. Since $G$ is 4-connected and $(H, t_1, t_2, t_3, t_4,t_5)$ is planar, $H$ must contain a
   path $P$ from $w'$ to $(P_4-\{w,w_4\})\cup (P_1-\{w,w_1\})\cup \{t_5\}$ and internally disjoint from
   $P_4\cup P_1\cup F$. If $P$ ends at $t_5$ then $P$ and three of $\{P_1,P_2,P_3,P_4\}$ show that $W$ is $(V(H\cap L),S)$-extendable.
   So by symmetry we may assume $P$ ends at $P_4-\{w,w_4\}$. Then replacing $P_4$ with the path in $P\cup (P_4-w_4)\cup ww'$ from $w$ to
   $t_4$, we obtain a contradiction to the minimality of $w_4Fw_1$.

  Note that $H$ has a path $R$ from $t_5$
       to $(P_4\cup w_4Fw_1\cup P_1)\setminus \{t_1, t_4\}$
       and internally disjoint from $P_1\cup P_4\cup w_4Fw_1$. For otherwise,
   $G$ has a 4-separation $(G_1',G_2')$ such that $V(G_1'\cap G_2')=\{t_1,t_2,t_3,t_4\}$, $W\subseteq G_1'$, and $G_2+t_5\subseteq G_2'$.
   By the choice of $(G_1,G_2)$, $|V(G_1')|=5$. This implies that $w_i=t_i$ for $i\in [4]$ and, hence, $t_1t_2\in E(H)$, a contradiction.

  We may assume that $H$ has a path from $t_5$ to $P_1\cup P_4$ and internally disjoint from $P_1\cup P_4\cup F$.
For, suppose not. Then, by planarity, there exist $a, b\in V(w_4Fw_1-\{w_1,w_4\})$ (not necessarily
distinct) such that $w_4,a,b,w_1$ occur on $F$ in clockwise order and all paths in $H$ from
$t_5$ to $P_1\cup P_4\cup (w_4Fw_1-aFb)$ must intersect $aFb$ first.
We choose $a, b$ so that $aFb$ is minimal.  Since $N_G(w)\cap V(w_4Fw_1-\{w_1,w_4\})=\emptyset$, $H$ has a separation
  $(H_1,H_2)$ such that $V(H_1\cap H_2)=\{a,b,t_5\}$, $aFb+t_5\subseteq
  H_1$, and $bFa+\{t_1,t_2,t_3,t_4\}\subseteq H_2$. Thus $V(H_1)=\{a,b,t_5\}$ as $G$ is 4-connected. Now, by the existence of $R$, $(i)$ holds with
$s:=t_5$.

   \medskip

  {\it Case} 1.  $H$ has paths from $t_5$ to both $P_1$ and $P_4$ and
  internally disjoint from $P_1\cup P_4\cup F$.

       Then $t_1,t_4\in S$ as otherwise we may reroute $P_1$ or $P_4$
       to $t_5$;  and the new path, $P_2$ and $P_3$, and $P_4$ or $P_1$ show that $W$ is $(V(H\cap L),S)$-extendable. So
       $S=\{t_1,t_4,t_5\}$. Let $v\in V(G_1\cap
       G_2)\setminus S$.

    We further choose $P_1,P_4$ so that, subject to the minimality of $w_4Fw_1$, the subgraph $K$ of $H$ contained in the closed region bounded by $(P_1-w)\cup (P_4-w)\cup t_4Dt_1\cup w_4Fw_1$
      is maximal. Then, every vertex in $V(P_1)\setminus
      \{w,w_1,t_1\}$ is cofacial with some vertex in $V(w_1Fw_2-w_1)\cup V(P_2-w)$; and
           every vertex in $V(P_4)\setminus \{w,w_4,t_4\}$  is cofacial with some vertex in $V(w_3Fw_4-w_4)\cup V(P_3-w)$.
   Let $T_1:=\{x\in  V(P_1\cup w_4Fw_1)\setminus\{t_1, w,w_4\}: x \mbox{ is cofacial with
            $t_4$}\}$ and $T_4:=\{x\in  V(P_4\cup w_4Fw_1)\setminus \{t_4,w, w_1\} : x \mbox{ is cofacial with $t_1$}\}$.
   Note that $t_1\notin T_4$ and $t_4\notin T_1$ by the existence of the path $R$.

We may assume  that $T_1=\emptyset$ or $T_4=\emptyset$. For otherwise, suppose $T_1\neq \emptyset$ and $T_4\neq \emptyset$.
Then let $a\in T_1$ and $b\in T_4$.  Now, by the existence of the path $R$, the vertices $w_4,a, b,w_1$ occur on $F$ in clockwise order.
Since $N_G(w)\cap V(w_4Fw_1-\{w_1,w_4\})=\emptyset$, $H$ has a separation $(H_1,H_2)$ such that $V(H_1\cap H_2)=\{t_1,t_4,a,b\}$,
$S\subseteq  V(H_1)$, and
   $\{w\}\cup (V(H\cap L)\setminus S)\subseteq V(H_2)$;
   hence, we have $(iii)$ with $s_1=t_1$ and $s_2=t_4$.

      We may also assume that if  $L-t_1$ has disjoint paths $R_3,R_2$ from $t_3,t_2$ to $t_4,v$, respectively,
         then $T_1\ne \emptyset$; for, otherwise,  $K-t_4$ contains a path $P$
    from $w_4$ to $t_5$ and internally disjoint from $w_4Fw_1\cup P_1$, and the paths $P_1,P_2\cup R_2, P_3\cup R_3,P\cup ww_4$ show that
    $W$ is $V(G_1\cap G_2)$-extendable in $G_1$.
    Similarly, we may assume that if $L-t_4$ has disjoint paths $Q_3,Q_2$ from  $t_3,t_2$ to $v,t_1$, respectively, then
    $T_4\ne \emptyset$.

Hence, since $T_1=\emptyset$ or $T_4=\emptyset$, $R_2, R_3$ do not
exist or $Q_2, Q_3$ do not exist.

   \medskip

    {\it Subcase} 1.1. $Q_2,Q_3$ or $R_2,R_3$ exist.

        Without loss of generality, assume that $R_2, R_3$ exist, and
        $Q_2,Q_3$ do not exist. Then $T_1\neq \emptyset$ and $T_4=\emptyset$. Since $Q_2,Q_3$ do not exist,
it follows from the choice of $(G_1,G_2)$ (minimality of $G_1$) that $v$ and $t_1Dt_2$ are cofacial in $G_1$.
Since  $T_4=\emptyset$,  there exists a path $P$ in $K-\{t_1, t_4\}$
from $w_1$ to $t_5$ and internally disjoint from $w_4Fw_1\cup P_4$. We choose
$P$ so that the subgraph $K'$ of $K$ in the closed region bounded by $P_1\cup P\cup t_5Dt_1$ is
maximal.

We may assume that  there exists a vertex $t\in V(t_1Dt_2-t_1)\cap V(P\cup
(w_1Fw_2-w_2))$. For, suppose not.
Then let $P_2'$ be a path in $P_2\cup t_1Dt_2$  from $w_2$ to $t_1$; now
       $P_3, P_4,P\cup ww_1$, and $P_2'\cup ww_2$ show that $W$ is $(V(H\cap L),S)$-extendable.

Suppose $t\in V(P)$. Choose $t$ so that $t_1Dt$ is maximal.
Then note that $t_1Dt=t_1P_1t$ (by the maximality of $K$) and, for any
vertex $t^*\in V(t_1P_1t)$, $t^*\notin T_1$ to avoid the 3-cut
$\{t^*,t_4,v\}$ in $G$. Thus,  since $T_1\ne \emptyset$, it follows from the maximality of $K'$ that $t$ is cofacial with
some vertex $t'\in V(w_4Fw_1-w_1)$. Choose $t'$ so that $t'Fw_1$ is
maximal. Now $t_4\ne  w_4$; otherwise,  $G$ has a 4-separation
 $(G_1',G_2')$ such that $V(G_1'\cap G_2')=\{t, t',t_4,v\}$,
 $G_1'\subseteq G_1-t_5$,
 $G_2+t_5\subseteq G_2'$, and $|V(G_1')|\ge 6$,
 contradicting the choice of $(G_1,G_2)$.
 Now, since $t, t'$ are cofacial and $t^*\notin T_1$ for any
vertex $t^*\in V(t_1Pt)$, it follows from planarity that   $T_1\subseteq V(w_4Ft'-w_4)$, and  we choose $u_1\in T_1$ with
$w_4Fu_1$  maximal.
 Thus  $G_1$ has a 5-separation $(H',L')$ such that $V(H'\cap L')=\{t',t,v,t_4,u_1\}$ is independent in $H'$, $\{w,w_2,w_3,w_4\}
          \subseteq V(H')\setminus V(L')$, and $L+\{t_1, t_5\}\subseteq
          L'$. Note that $|V(H'\cap L')\cap
        V(G_1\cap G_2)|<|S|$ (since $t\ne t_1$). So by the choice of $(H,L)$, $H'$ contains no $V(H'\cap L')$-good wheel.
        Hence by Lemma~\ref{5cutobs},  $(H',V(H'\cap L'))$ must be the 9-vertex graph in Figure~\ref{obstructions}. However, this
        is impossible as  $w_4w_2\notin E(H')$ but $w_4$ is the unique neighbor of $u_1$ in $H'-L'$.

Thus, $t\in V(w_1Fw_2)\setminus \{w_1,w_2\}$ for all choices of $t$.
Choose $t$ so that $w_1Ft$ is minimal.
Now, $V(P_1\cap t_1Dt)=\{t_1\}$ and, by the maximality of $K$, each vertex of $P_1$ is cofacial with some vertex in $V(w_1Ft-w_1)$.

Suppose there exists $u_1\in V(P_1-w_1)\cap T_1$. Then there exists $u_1'\in
V(w_1Ft-w_1)$ such that $u_1'$ and $u_1$ are cofacial. Choose such
$u_1,u_1'$  that $u_1P_1t_1$ and $u_1'Ft$ are minimal. Suppose there exists $w'\in N_G(w)\cap V(u_1'Ft-\{u_1',t\})$.
Then since $G$ is 4-connected, it follows from the
choice of $u_1,u_1'$ that $H$ has a path $P_1'$ from $w'$ to $t_1$ and internally disjoint from $F\cup P$. Now $P_1'\cup ww',
P_2\cup R_2, P_3\cup R_3,  P\cup ww_1$ show that $W$ is $V(G_1\cap G_2)$-extendable.
So we may assume $N_G(w)\cap V(u_1'Ft-\{u_1',t\})=\emptyset$.
Then  $G_1$ has a 5-separation $(H',L')$ such that $V(H'\cap L')=\{t,v, t_4,u_1,u_1'\}$ is independent in $H'$,
$\{w,w_1,w_2,w_3\}\subseteq V(H')\setminus V(L')$,
and $L+\{t_1,t_5\} \subseteq L'$.  Note that $w_1w_2,w_1w_3\notin E(H')$; so $(H', V(H'\cap L'))$ cannot be any graph in Figure~\ref{obstructions}.
Thus, by  Lemma~\ref{5cutobs}, $H'$ has a
$V(H'\cap L')$-good wheel. Hence,  $(H',L')$  contradicts
the choice of $(H,L)$ as  $|V(H'\cap L')\cap V(G_1\cap G_2)|<|S|$.

Thus, we may assume $V(P_1-w_1)\cap T_1=\emptyset$. So there exists $u_1\in
V(w_4Fw_1-w_4)\cap T_1$. Choose $u_1$ so that $u_1Fw_1$ is minimal.

If $t_4=w_4$ then $G$ has a
4-separation $(G_1',G_2')$ such that $V(G_1'\cap
G_2')=\{t,v,w_4,w\}$, $G_1'\subseteq G_1-t_5$,
$G_2\subseteq G_2'$, and $|V(G_1')|\ge 6$; which
contradicts the choice of $(G_1,G_2)$.
So $t_4\ne w_4$. Then $G_1$ has a 5-separation $(H',L')$ such that $V(H'\cap L')=\{t,v, t_4,u_1,w\}$ is independent in $H'$,
$\{w_2,w_3,w_4\}\subseteq V(H')\setminus V(L')$,
and $L+\{t_1,t_5\} \subseteq L'$. Now $H'$ has no $V(H'\cap
L')$-good wheel; otherwise, $(H',L')$  contradicts
the choice of $(H,L)$ as $|V(H'\cap L')\cap V(G_1\cap G_2)|<|S|$.
Hence, by Lemma~\ref{5cutobs}, $(H',V(H'\cap L')$ is the 8-vertex or 9-vertex graph in Figure~\ref{obstructions}. Note that $w$ is adjacent to all of $w_2, w_3, w_4$.
Thus,  $(H',V(H'\cap L')$ must be the 8-vertex graph in Figure~\ref{obstructions}.
However, this forces $t_2=w_2$ and $t_3=w_3$; so $t_2t_3\in E(W)\subseteq E(H)$, a contradiction as $V(H\cap L)$ is independent in $H$.

\medskip

         {\it Subcase} 1.2 Neither $Q_2,Q_3$ nor  $R_2,R_3$ exist.

           Then, by the choice of $(G_1, G_2)$, we see that $t_1Dt_2$ and $v$ are cofacial, and that $t_3Dt_4$ and $v$
           are cofacial. Moreover, since $G$ is 4-connected, $\{t_2,
           t_3, v\}$ is not a cut in $G$. Hence,
           $V(L)=\{t_1,t_2,t_3,t_4,t_5, v\}$ and,  by the choice of
           $(G_1,G_2)$,  we have $vt_2,vt_3\in E(G)$.

           Suppose there exist $a\in V(t_1Dt_2)\cap V(w_1Fw_2-w_2)$ and
          $b\in V(t_3Dt_4)\cap V(w_3Fw_4-w_3)$. Then  $G$ has a 4-separation $(G_1',G_2')$ such that $V(G_1'\cap G_2')=\{a,b,v,w\}$, $\{w_2,w_3\}\subseteq V(G_1')\setminus V(G_2')$,
          and $G_2+\{t_1,t_4,t_5\}\subseteq G_2'$. Now $(G_1',G_2')$
          contradicts the choice of $(G_1,G_2)$.

         So by symmetry, we may assume that  $t_3Dt_4\cap (w_3Fw_4-w_3)= \emptyset$. Thus, $t_3Dt_4\cup P_3$ contains a path $P$ from $w_3$ to $t_4$
        and internally disjoint from $F$. If $H$ has a path $Q$ from $w_4$ to $t_5$ and internally disjoint from $P\cup w_3Fw_1\cup P_1$ then
           $P_1,P_2, P\cup ww_3, Q\cup ww_4\}$
         show that $W$ is  $(V(H\cap L),S)$-extendable. So assume that $Q$ does not exist. Then there exist $x\in V(P)\cup V(w_3Fw_4-w_3)$ and
         $y\in V(P_1)\cup V(w_4Fw_1-w_4)$  such that $x$ and $y$ are cofacial.  By the choice of $P$ and planarity of $H$, $x\in V(t_3Dt_4)\cap V(P_4-w_4)$.
        Choose $x$ to minimize $xDt_4$.

        First, suppose $t_1Dt_2\cap (w_1Fw_2-w_2)\ne \emptyset$ and $y\in
        V(w_4Fw_1-w_4)$ for all choices of $y$.  Choose $a\in V(t_1Dt_2)\cap V(w_1Fw_2-w_2)$ such that $w_1Fa$ is minimal,
       and  choose $x,y$ so that $yFw_1$ is minimal. Then $G_1$ has a 5-separation $(H',L')$ such that $V(H'\cap L')=\{a,v,x,y,w\}$ is independent in $H'$,
       $\{w_2,w_3,w_4\}\subseteq V(H')\setminus V(L')$, and
         $L+\{t_1,t_4,t_5\}\subseteq L'$. Since $|V(H'\cap L')\cap V(G_1\cap G_2)|<|S|$, we see that $H'$ has no $V(H'\cap L')$-good wheel. Hence, by Lemma~\ref{5cutobs},
         $(H',V(H'\cap L'))$ is the 8-vertex or 9-vertex graph in Figure~\ref{obstructions}. Since $w$ is adjacent to all of $w_2,w_3,w_4$, we see that
         $|V(H')|=8$ which forces $t_2=w_2$ and $t_3=w_3$. Hence,
         $t_2t_3\in E(W)\subseteq E(H)$, a contradiction as $V(H\cap   L)$ is independent in $H$.

      Now suppose $t_1Dt_2\cap (w_1Fw_2-w_2)\ne \emptyset$ and
         $y\in V(P_1-w_1)$ for some choice of $y$.  Choose such $y$ so that $yP_1t_1$ is minimal; so $K$ has  a path $P_5$ from $y$ to $t_5$ and internally
         disjoint from $P_1\cup P_4$.   Let $a\in V(t_1Dt_2)\cap  V(w_1Fw_2-w_2)$ such that $w_1Fa$ is minimal.
          Note that $y\notin V(t_1Dt_2)$ (to avoid the 3-cut $\{v,x,y\}$ in $G$).   So $y$ is not cofacial with $P_2-w_2$ and,
        by the maximality of $K$, there exists $y'\in V(w_1Fa-w_1)$ such that $y$ and $y'$ are cofacial.
         We choose $y'$ so that $y'Fa$ is minimal.
         If there exists $w'\in N_G(w)\cap V(y'Fa-\{y',a\})$ then
        by the minimality of $yP_1t_1$ and $y'Fa$ and by the 4-connectedness of $G$,
         $H$ has a path $S$ from $w'$ to $t_1$ and
        internally disjoint from $F\cup P_5$; now $S\cup ww', P_2, P\cup ww_3, wP_1y\cup P_5$ show that $W$ is $(V(H\cap L),S)$-extendable.
         Hence, we may assume $N_G(w)\cap V(y'Fa)=\emptyset$.
         So $G_1$ has a 5-separation $(H',L')$ such that $V(H'\cap L')=\{a,v,x,y,y'\}$ is independent in $H'$,
         $\{w,w_1,w_2,w_3,w_4\}\subseteq V(H')\setminus V(L')$,
          and $L+\{t_1,t_4,t_5\}\subseteq L'$.   By Lemma~\ref{5cutobs}, $H'$ contains a $V(H'\cap L')$-good wheel. Now $(H',L')$ contradicts the choice of $(H,L)$, as
         $|V(H'\cap L')\cap V(G_1\cap G_2)|<|S|$.

          Hence, we may assume that  $t_1Dt_2\cap (w_1Fw_2-w_2)=
          \emptyset$. Then $t_1Dt_2\cup P_2$ has a path $Q$ from $w_2$
          to $t_1$ and internally disjoint from $F$. Similar to the
          argument for  showing the existence of $x$ and $y$ above, we
          may assume that there exist $p\in V(t_1Dt_2)\cap V(P_1-w_1)$
    and $q\in V(P_4)\cup V(w_4Fw_1-w_1)$ such that $p$ and $q$ are cofacial.
            Then $w_4,y,q,w_1$ occur on $F$ in clockwise order; for, otherwise, by planarity of $G_1$, $G$ would have a 3-cut consisting of $v$, one of $\{x,q\}$,
          and one of $\{p,y\}$.

So $G_1$ has a 5-separation $(H',L')$ such that $V(H'\cap L')=\{p,v,x,y,q\}$ is independent in $H'$,
           $\{w,w_1,w_2,w_3,w_4\}\subseteq V(H')\setminus V(L')$, and $L+\{t_1,t_4,t_5\}\subseteq L'$. By Lemma~\ref{5cutobs}, $H'$ has a $V(H'\cap L')$-good wheel. If $\{x,p\}\ne
           \{t_1,t_4\}$ then
          $|V(H'\cap L')\cap V(G_1\cap G_2)|<|S|$; and hence
          $(H',L')$ contradicts the choice of $(H,L)$. So $x=t_4$ and
           $p=t_1$; hence  $(iii)$ holds.

  \medskip

  {\it Case} 2.  Case 1 does not occur.

We choose $P_1,P_4$, subject to the minimality of $w_4Fw_1$, to maximize
the subgraph $K$ of $H$ contained in the closed region bounded by $(P_1-w)\cup (P_4-w)\cup t_4Dt_1\cup w_4Fw_1$.

Without loss of generality, we may assume that
$G_1$ has no path from $t_5$ to $P_4$ and internally disjoint from $P_4\cup P_1\cup w_4Fw_1$, but
 $G_1$ has a path $P_5'$ from $t_5$ to $P_1$ and internally disjoint from $P_4\cup P_1\cup w_4Fw_1$.
Then $t_4Dt_5\cap ((w_4Fw_1-w_4)\cup P_1)\ne \emptyset$.
Moreover, we may assume $t_1\in S$; as otherwise we could reroute $P_1$ to end at $t_5$ which,
 along with $P_2,P_3,P_4$, shows that $W$ is $(V(H\cap L),S)$-extendable.

\medskip

{\it Subcase} 2.1.  $t_4Dt_5\cap (w_4Fw_1-w_4)=\emptyset$ and  $t_1Dt_2\cap (w_1Fw_2-w_2)= \emptyset$.

Then there exists $a\in V(t_4Dt_5)\cap V(P_1-w_1)$,
and we choose such $a$ with $t_1P_1a$ minimal. Note $t_1\ne a$ by the existence of $R$.  Let $b\in V(t_1Dt_2)\cap V(P_1-w_1)$ with $t_1P_1b$ maximal.
By the maximality of $K$, $t_1Db=t_1P_1b$.
If $a\in V(t_1P_1b)$ then $G$ has a 4-separation $(G_1',G_2')$ such that $V(G_1'\cap G_2')=\{a,t_2,t_3,t_4\}$, $\{w,w_1\}\subseteq V(G_1')\setminus V(G_2')$, and
$G_2+\{t_1,t_5\}\subseteq G_2'$, which  contradicts the choice of $(G_1,G_2)$.

Hence, $a\in V(bP_1w_1)\setminus \{b,w_1\}$. Let $P_5:=wP_1a\cup aDt_5$. We consider the paths $P_2,P_3,P_4,P_5$. By the minimality of $t_1P_1a$, we see that $t_1P_1a$ is a path
from $t_1$ to $P_5$ and internally disjoint
from $P_5\cup P_2\cup F$. Since  $t_1Dt_2\cap (w_1Fw_2-w_2)= \emptyset$, $t_1Dt_2$ contains a path from $t_1$ to $P_2$ and internally disjoint from $P_5\cup P_2\cup F$. Hence,
we are back to Case 1 (with $t_5,t_1,t_2,t_3,t_4$ as $t_4,t_5,t_1,t_2,t_3$, respectively).

\medskip

{\it Subcase} 2.2. Either $t_4Dt_5\cap (w_4Fw_1-w_4)\ne \emptyset$ or  $t_1Dt_2\cap (w_1Fw_2-w_2)\ne \emptyset$.

First, we may assume that $t_4Dt_5\cap (w_4Fw_1-w_4)\ne \emptyset$. For, if not, then
there exists $a\in V(P_1-w_1)\cap V(t_4Dt_5)$ and choose $a$ so that $t_1P_1a$ is minimal. Moreover,
$t_1Dt_2\cap (w_1Fw_2-w_2)\ne \emptyset$; so $H$ has no path from $t_1$ to $P_2$ and internally disjoint from $F$.
Note that $t_1\ne a$ by the path $R$.
Let $P_5:=wP_1a\cup aDt_5$.
Now consider the paths $P_2,P_3,P_4,P_5$.  We see that $t_1P_1a$ is a path
from $t_1$ to $P_5$ and internally disjoint
from $P_5\cup P_2\cup F$.
Hence, since $t_1Dt_2\cap (w_1Fw_2-w_2)\ne \emptyset$, we could take the mirror image of $G_1$ and view  $t_2,t_1,t_5, t_4,t_3$ as
$t_4,t_5,t_1,t_2,t_3$, respectively;  and, thus, may assume $t_4Dt_5\cap (w_4Fw_1-w_4)\ne  \emptyset$.

Then  $t_1Dt_2\cap (w_1Fw_2-w_2)= \emptyset$, and we let $a\in  V(t_4Dt_5)\cap V(w_4Fw_1-w_4)$ with $w_4Fa$ minimal. Let $t\in V(t_1Dt_2\cap P_1)$ with $t_1Dt$ maximal.
Then $t\ne w_1$ and,  by the maximality of $K$, $t_1Dt=t_1P_1t$.

Note that $t_4Dt_5\cap t_1Dt_2=\emptyset$.  For,  otherwise, let $p\in V(t_4Dt_5)\cap V(t_1Dt_2)$. Then $G$ has a 4-separation
$(G_1',G_2')$ with $V(G_1'\cap G_2')=\{p,t_2,t_3,t_4\}$, $w,w_1\in V(G_1'-G_2')$, and $G_2+\{t_1,t_5\}\subseteq G_2'$.
Clearly, $(G_1',G_2')$ contradicts the choice of $(G_1,G_2)$.

If  there exist $c\in V(t_1Dt)$ and $b\in V(aFw_1-w_1)$
such that $b$ and $c$ are cofacial, then $(iv)$ holds. So assume such $b,c$ do not exist. Then
$K$ contains a path $P$ from $w_1$ to $t_5$ and internally disjoint from $F\cup  t_1Dt_2$.
By the existence of a path in $P_1$ from $t_1$ to $P$ and the
path $t_1Dt_2$, we are back to Case 1 (with $t_5,t_1,t_2, t_3, t_4$ playing the roles of $t_4,t_5,t_1, t_2, t_3$, respectively). $\Box$

\medskip

{\it Subcase} 2.3. $t_4Dt_5\cap (w_4Fw_1-w_4)\ne \emptyset$ and  $t_1Dt_2\cap (w_1Fw_2-w_2)\ne \emptyset$.

Let $a\in V(t_4Dt_5)\cap V(w_4Fw_1-w_4)$ and $b\in V(t_1Dt_2)\cap V(w_1Fw_2-w_2)$, and we choose $a,b$ to minimize $aFb$.
  Consider the separation $(H_1,H_2)$ in $G_1$ such that  $V(H_1\cap
  H_2)=\{a,b,w\}$, $V(H_1)\cap \{t_i:i\in [5]\}= \{t_1,t_5\}\subseteq
  S$, and $bFa+\{t_2,t_3,t_4\}\subseteq H_2$.

\begin{itemize}
  \item [(1)]  $|N_G(w)\cap V(aFb)|\ge 2$.
\end{itemize}
For, $|N_G(w)\cap V(aFb)|=1$.
If $w_1\notin \{a,b\}$  then we have $(ii)$.
 So by symmetry, assume $w_1=b$. Consider the 5-separation $(H',L')$ in $G_1$ such that $V(H'\cap L')=\{a,b, t_2,t_3,t_4\}$ is independent in $H'$,
 $bFa+w\subseteq H'$, and $L\cup H_1\subseteq L'$. By the choice of $(H,L)$, $H'$ has no $V(H'\cap L')$-good wheel. So by
Lemma~\ref{5cutobs}, $(H',V(H'\cap L'))$  is one of the graphs in Figure~\ref{obstructions},

  Suppose $w_3\ne t_3$. Then $w_4\ne t_4$ to avoid the 4-separation $(G_1',G_2')$ with
 $V(G_1'\cap G_2')=\{b,t_2,t_3,w_4\}$, $\{w,w_3\}\subseteq V(G_1'-G_2')$, and $G_2+\{t_1,t_5\}\subseteq G_2'$. So  $ww_3w_4w\subseteq H'-L'$, and
$(H',V(H'\cap L'))$ must be the  9-vertex graph in Figure~\ref{obstructions}. However,  this is impossible, as $w_4b\notin E(H')$ and one of the following holds:
 $w_4$ is the  unique neighbor of $a$ in $H'-L'$,  or $w$ is the unique neighbor of $b$ in $H'-L'$.

 Therefore, $w_3=t_3$. Then $t_2\ne w_2$; for, otherwise,  $w_2Fw_3=t_2t_3\in E(H)$ as $G$ is 4-connected, a
 contradiction. Similarly, $t_4\ne w_4$.  Thus, $w_2ww_4\subseteq H'-L'$. So by Lemma~\ref{5cutobs}, $(H',V(H'\cap L'))$ must be the 8-vertex or
 9-vertex graph in Figure~\ref{obstructions}. But this is not possible, as $w_4b\notin E(H')$ and one of the following holds:
$w_4$ is the unique neighbor of $a$ in $H'-L'$, or $w$ is the unique neighbor of $b$ in $H'-L'$. $\Box$

\medskip

We may assume $t_1\ne w_1$; since otherwise $b=t_1$ by the minimality of $aFb$, and we would have $|N_G(w)\cap V(aFb)|=1$, contradicting (1).
We may also assume
\begin{itemize}
\item [(2)]  $N_G(w)\cap V(aFb)\ne \{a,b\}$.
\end{itemize}
For, otherwise, $a,b\in N_G(w)$ (so $w_1=a$) and $G$ has a 4-separation $(G_1',G_2')$ such that $V(G_1'\cap G_2')=\{a,b,t_1,t_5\}$,
$aFb\subseteq G_1'$, and $G_2\cup bFa+w\subseteq G_2'$. Hence, by the choice of $(G_1,G_2)$, $|V(G_1')|\le 5$.

If $|V(G_1')|=4$ then $P_1=wat_1$; so
$bt_1\in E(G)$ (by the minimality of $aFb$) and $at_5\in E(G)$ (by the path $R$),  and, hence, $wat_5,wbt_1$ and two of $P_2,P_3,P_4$ show that $W$ is $(V(H\cap L),S)$-extendable.

Hence, we may assume $|V(G_1')|=5$ and let $u\in V(G_1')\setminus V(G_2')$. Then $N_G(u)=\{a, b, t_1,t_5\}$. Since $a=w_1$, $u\notin V(W)$ and $P_1=waut_1$.
If $t_5a\in E(G)$ then $wat_5, wbut_1$, and two of $P_2,P_3,P_4$
show that  $W$ is $(V(H\cap L),S)$-extendable. If $t_1b\in E(G)$ then $waut_5, wbt_1$, and two of $P_2,P_3,P_4$ show that $W$ is $(V(H\cap L),S)$-extendable. So assume $t_5a,t_1b\notin E(G)$. Then
$G$ has a 4-separation $(G_1'',G_2'')$ such that $V(G_1''\cap G_2'')=\{t_2,t_3,t_4,u\}$,
$\{a,b,w\}\subseteq V(G_1'')\setminus V(G_2'')$, and $G_2+\{t_1,t_5\}\subseteq G_2'$. Hence, $(G_1'',G_2'')$  contradicts the choice of $(G_1,G_2)$.  $\Box$

\medskip

Now consider the 5-separation $(H',L')$ in $G_1$ with $V(H'\cap L')=\{a,b,w,t_1,t_5\}$, $bFa\subseteq L'$, and $aFb+\{t_1,t_5\}\subseteq H'$. Note that $|V(H'\cap L')\cap V(G_1\cap G_2)|\le |S|$ and
$H'\subseteq H$ but $H'\ne H$; so by the choice of $(H,L)$, $H'$ has no $V(H'\cap L')$-good wheel. Thus, since $N_G(w)\cap V(aFb)\ne \{a,b\}$, $(H', V(H'\cap L'))$ is one of the graphs in
Figure~\ref{obstructions}. Recall that $t_5t_1\notin E(H)$ as $V(H\cap L)$ is independent in $H$.

First, suppose $|V(H')|=6$ and let $u\in V(H')\setminus V(L')$. Then by (1) and (2), $aFb=aub$ and $u\in N_G(w)$.  By the minimality of $aFb$, $t_1b\in E(G)$. If $u=w_1$ then $t_5u\in E(G)$ (because of $R$) and $wb\in E(G)$; so
$wut_5,wbt_1$ and two of $P_2,P_3,P_4$ show that $W$ is $(V(H\cap L),S)$-extendable. Hence, we may assume $a=w_1$. Then $at_5,at_1\in E(G)$ (because of $P_1$ and $R$) and, hence, $ut_1\in E(G)$;
so $wat_5,wut_1$ and two of $P_2,P_3,P_4$ show that $W$ is $(V(H\cap L),S)$-extendable.

Now assume $|V(H')|=7$.
First, suppose $|V(aFb)|\ge 4$ and let $aFb=auvb$.   If $P_1=wat_1$ then $G$ has a separation $(G_1',G_2')$ such that $V(G_1'\cap G_2')=\{a,t_1,b,w\}$, $\{u,v\}\subseteq
V(G_1'-G_2')$, and $G_2+t_5\subseteq G_1$; and  $(G_1',G_2')$ contradicts the choice of $(G_1,G_2)$.
If $P_1=wvt_1$ then $wa,wu\notin E(H)$; so $ut_1,ut_5\in E(H)$,
contradicting the existence of the path $P_5'$.  So $P_1=wut_1$ then
$t_5u\in E(H)$ (by $P_5'$) and $vw,vt_1\in E(H)$ (by 4-connectedness of $G$); so $wut_5, wvt_1$ and two of $P_2,P_3,P_4$ show that $W$ is $(V(H\cap L),S)$-extendable.
So we may assume $|V(aFb)|=3$ and let $aFb=aub$ and $v\in V(H')\setminus (V(L')\cup \{u\})$.  Then $wu\in E(G)$ by (1) and (2).
If $t_5u\in E(G)$ then $N_G(v)=\{b,t_1,t_5,u\}$ and $t_5a\in E(G)$ (by the minimality of $aFb$); now
$wat_5,wuvt_1$ (when $wa\in E(G)$) or $wut_5, wbvt_1$ (when $wb\in E(G)$), and two of $P_2,P_3,P_4$ show that $W$ is $(V(H\cap L),S)$-extendable. So assume $t_5u\notin E(G)$. By the same argument, we may assume $t_1u\notin E(G)$.
 Then $t_1v,t_5v\in E(G)$. Note that $t_1b\in E(G)$ or $t_5a\in E(G)$; otherwise, $(H-\{t_1,t_5\}, G_2\cup L\cup t_1vt_5)$ is a 4-separation in $G$ contradicting the choice of $(G_1,G_2)$. So by symmetry, we may assume
$t_5a\in E(G)$. If $wa\in E(G)$ then $wat_5, wuvt_1$, and  two of $P_2,P_3,P_4$ show that $W$ is $(V(H\cap L),S)$-extendable. So assume $wa\notin W(G)$; hence,
$wb\in E(G)$ by (1). If  $t_1b\in E(G)$ then
$wbt_1, wuvt_5$, and two of $P_2,P_3,P_4$ show that $W$ is $(V(H\cap L),S)$-extendable. So assume $t_1b\notin E(G)$. Now $G_1$ has a 5-separation $(H^*,L^*)$
such that $H^*=(H-t_1)-t_5v$ and $L^*=L\cup t_1vt_5$. Note that $V(H^*\cap L^*)$ is independent in $H^*$ and $W$ is $V(H^*\cap L^*)$-good. So $(H^*,L^*)$ contradicts the choice of $(H,L)$ as
$|V(H^*\cap L^*)\cap V(G_1\cap G_2)|<|S|$.

Suppose $|V(H')|=8$ and let $H'-L'=xyz$. Note that exactly one vertex  in  $V(H'\cap L')$ is adjacent to all of $\{x,y,z\}$, and call that vertex $t$. If $t=w$ then we may let $aFb=axyzb$; we see that $wxt_5, wzt_1$ and
 two of $P_2,P_3,P_4$ show that $W$ is $(V(H\cap L),S)$-extendable. If $t=t_5$ then we may let $aFb=axyb$; we see that $wxt_5, wyzt_1$,  and two of $P_2,P_3,P_4$ show that $W$ is $(V(H\cap L),S)$-extendable. Similarly, if $t=t_1$
then $W$ is $(V(H\cap L),S)$-extendable. Now assume $t=a$; the argument for $t=b$ is symmetric. Then we may let $aFb=axb$.
   If $wa\in E(H)$ then $wazt_5, wxyt_1$, and
 two of $P_2,P_3,P_4$ show that $W$ is $(V(H\cap L),S)$-extendable. So assume $wa\notin E(H)$. Then $wb\in E(H)$ by (1).
If $t_1b\in E(H)$ then $wxyt_5, wbt_1$, and two of $P_2,P_3,P_4$ show that $W$ is $(V(H\cap L),S)$-extendable. So $t_1b\notin E(H)$.
Let $H^*=(H-t_5)-t_1z$ and $L^*=L\cup t_1zt_5$. Note that $V(H^*\cap L^*)$ is independent in $H^*$ and
$W$ is $V(H^*\cap L^*)$-good. So $(H^*,L^*)$ contradicts the choice of $(H,L)$ as
$|V(H^*\cap L^*)\cap V(G_1\cap G_2)|<|S|$.


Finally, assume $|V(H')|=9$. Let $V(H'-L')=\{u,x,y,z\}$ such that $xz\notin E(H)$, and $u$ is the unique neighbor of some vertex $t\in V(H'\cap L')$.
If $t=w$ then we see that $aFb=aub$ and let $ax,zb\in E(H)$; now $waxt_5, wuyt_1$ (when $wa\in E(H)$) or  $wuxt_5,wbzt_1$ (when $wb\in E(H)$),
and  two of $P_2,P_3,P_4$ show that $W$ is $(V(H\cap L),S)$-extendable.
If $t=a$ then we may let $aFb=auxb$; then $wut_5,wxyzt_1$,  and two of $P_2,P_3,P_4$ show that $W$ is $(V(H\cap L),S)$-extendable. If  $t=b$ then may let $aFb=axub$; then
$wxyt_5,wut_1$, and two of $P_2,P_3,P_4$ show that $W$ is $(V(H\cap L),S)$-extendable. If $t=t_5$ then we may let $aFb=axyb$; now $wxut_5,wyzt_1$ and two of $P_2,P_3,P_4$ show that $W$ is $(V(H\cap L),S)$-extendable.
If $t=t_1$ then we may let $aFb=ayxb$; now $wyzt_5, wxut_1$, and two of $P_2,P_3,P_4$ show that $W$ is $(V(H\cap L),S)$-extendable.
\end{proof}

Next, we eliminate the possibility $(iv)$ of Lemma~\ref{extension-4cut1} by
working with more than one wheels.

\begin{lem}\label{extension-4cut1}
With the same assumptions of Lemma~\ref{extension-4cut}, $H$ has a $(V(H\cap L),S)$-extendable wheel, or
$G_1$ has a $V(G_1\cap G_2)$-extendable wheel, or $(i)$ or $(ii)$ or $(iii)$ of Lemma~\ref{extension-4cut} holds for any $w\in V(H-L)$ and
for any $V(H\cap L)$-good wheel $W$ in $H$ with center $w$.
\end{lem}

\begin{proof}
Suppose $(iv)$ of Lemma~\ref{extension-4cut} holds for some $V(H\cap L)$-good wheel $W$ with center $w$. Then
there exist $a,b\in V(W-w)\setminus N_G(w)$, $c\in V(H)\setminus V(W)$,
                      and  separation $(H_1,H_2)$ in $H$ such that
                      $V(H_1\cap H_2)=\{a,b,c\}$,  $|V(H_1)\cap
                      V(H\cap L)|=2$, $V(H_1)\cap V(H\cap L)\subseteq
                      S$, and $(N_G(w)\cup \{w\})\cap V(H_1)=\emptyset$.
       Let $G_1$ be drawn in a closed disc in the plane with no edge crossings such that $V(G_1\cap G_2)$ is contained in the boundary of that disc.
Let $V(H\cap L)=\{t_i: i\in [5]\}$ and we may assume that $(H,t_1,t_2,t_3,t_4,t_5)$ is planar.
  Recall from the assumptions in Lemma~\ref{extension-4cut} that $(H,L)$ is chosen to minimize $|S|$, where $S:=V(G_1\cap G_2)\cap V(H\cap L)$.
Without loss of generality, we may assume that $V(H_1)\cap V(H\cap L)=\{t_1,t_5\}$. So $t_1,t_5\in S.$

By Lemma~\ref{extension-5cut}, $W$ is $V(H\cap L)$-extendable in $H$.
So there are four paths $P_i, i\in [4]$, in $H$ from $w$ to $\{t_i: i\in [5]\}$, such that
$V(P_i\cap P_j)=\{w\}$ for $i\ne j$, $|V(P_i)\cap W|=2$ for $i\in [4]$,
$|V(P_i)\cap \{t_j: j\in [5]\}|=1$ for $i\in [4]$. Without loss of generality, we may assume that $t_i\in V(P_i)$ for $i\in [4]$.
Note that $P_2,P_3,P_4$ are disjoint from $H_1$, and $P_1\cap H_1=cP_1t_1$. We further choose $a, b,c$ so that
$aFb$ and $cP_1t_1$ are minimal.

Now $t_3\notin S$. For, suppose $t_3\in S$. Then, $t_3$ is cofacial with $t_1$ or $t_5$.  If $t_3$ is cofacial with $t_5$ then
$G$ has a 4-separation $(G_1',G_2')$ such that $V(G_1'\cap G_2')=\{t_1, t_2, t_3,t_5\}$, $W\subseteq G_1'$, and $G_2+t_2\subseteq G_2'$;
 which contradicts the choice of $(G_1,G_2)$.  We derive a similar
 contradiction if $t_3$ is cofacial with $t_1$, using the cut $\{t_1,t_3,t_4,t_5\}$.

Let $F=W-w$ and let $D$ denote the outer walk of $H$.  We choose
$P_1,P_2,P_3,P_4$ so that the following are satisfied in the order listed: $w_4Fw_1$ is minimal,   $w_3Fw_2$ is
minimal, and the subgraph $K$ of $H$ contained inside the region
bounded by $P_4\cup t_4Dt_1\cup P_1$ is minimal.
Then every vertex of $P_4$ is cofacial with a vertex in $w_4Fa-w_4$,
every vertex of $P_1$ is cofacial with a vertex in $bFw_1-w_1$, and

\begin{itemize}
\item [(1)]  $N_G(w)\cap V(w_3Fw_2)=\{w_1,w_2,w_3,w_4\}$.
\end{itemize}
For, suppose (1) fails and let $w'\in N_G(w)\cap V(w_3Fw_2)\setminus
\{w_1,w_2,w_3,w_4\}$.
First, assume $w'\in V(w_4Fw_1)\setminus \{w_1,w_4\}$. If $w'\in
V(w_4Fa-w_4)$ then since $G$ is 4-connected, $K$ has a path $P$ from $w'$
to $P_4$ and internally disjoint from $P_4\cup F$. Hence, we can
replace $P_4$ by a path in $P\cup (P_4-\{w,w_4\})$ from $w$ to $t_4$,
contradicting the minimality of $K$. We get the same contradiction if
$w'\in V(bFw_1-w_1)$.

Now assume  $w'\in V(w_3Fw_4)\setminus
\{w_3,w_4\}$. Consider the subgraph $J$ of $H$ contained in the closed
region bounded by $P_3\cup t_3Dt_4\cup P_4$. By the minimality of
$w_3Fw_2$, $J$ has no path from $w'$ to $t_3$ and internally disjoint
from $F\cup P_4$. Thus, there exist $x\in V(w_3Fw'-w')$ and $y\in
V(w'Fw_4-w')\cup V(P_4)$ such that $x,y$ are cofacial. Since $G$ is
4-connected, $y\in V(P_4-w_4)$. Note that $y$ is cofacial with some
vertex on $w_4Fa-w_4$, say $z$. Then $G$ has a 4-separation
$(G_1',G_2')$ such that $V(G_1'\cap G_2')=\{w,x,y,z\}$, $w'Fw_4\subseteq G_1'-G_2'$,
and $G_2+\{t_1,t_2,t_3,t_5\}\subseteq G_2'$. However, $(G_1',G_2')$
contradicts the choice of $(G_1,G_2)$.

Similarly, if  $w'\in V(w_1Fw_2)\setminus \{w_1,w_2\}$ then we derive a contradiction. $\Box$

\begin{itemize}
  \item [(2)] $w_i\ne t_i$ for $i\in \{2,3,4\}$.
\end{itemize}
First, $w_3\ne t_3$. For, suppose $w_3=t_3$. Then $w_2\ne t_2$ as,
otherwise, since $G$ is 4-connected,  $w_2Fw_3=t_2t_3\in E(W)\subseteq E(H)$, a contradiction.
Now by (1), $G_1$ has a 5-separation $(H',L')$ such that $V(H'\cap L')=\{b,c,t_2,w_3,w_4\}$ is independent in $H'$,
$ww_1w_2w\subseteq H'-L'$, and $L+\{t_1,t_4,t_5\}\subseteq L'$. By the choice of $(H,L)$, $H'$ has no $V(H'\cap L')$-good wheel. So by
Lemma~\ref{5cutobs}, $(H',V(H'\cap L'))$ is the 9-vertex graph in Figure~\ref{obstructions}.
This is not possible, as $w$ is the unique neighbor of $w_4$ in $H'-L'$ and $wb\notin E(H')$.

Next, $w_4\ne t_4$. For, suppose $w_4=t_4$. Then by (1),
 $G_1$ has a 5-separation $(H',L')$ such that $V(H'\cap L')=\{b,c,t_2,t_3,w_4\}$ is independent in $H'$,
$w_1ww_3\subseteq H'-L'$, and $L+\{t_1,t_5\}\subseteq L'$. By the choice of $(H,L)$, $H'$ has no $V(H'\cap L')$-good wheel. So by
Lemma~\ref{5cutobs}, $(H',V(H'\cap L'))$ is the 8-vertex or 9-vertex graph in Figure~\ref{obstructions}.
Now $|V(H')|=9$; as otherwise $w_2=t_2$ is adjacent to all vertices in
$H'-L'$, which implies $bw\in E(H')$, a contradiction.  Let $v\in V(H'-L')\setminus
\{w,w_1,w_3\}$. Since $w_1w_3\notin E(H)$, $w$ and $v$ both have
degree 3 in $H'-L'$. Therefore, $v=w_2$ is the unique neighbor
of $t_2$ in $H'-L'$, which implies $wb\in E(H)$, a contradiction.

Now $w_2\ne t_2$. For, suppose $w_2=t_2$.
Then $G_1$ has a 5-separation $(H',L')$ such that $V(H'\cap L')=\{a,w_1,w_2,t_3,t_4\}$ is independent in $H'$,
$ww_3w_4w\subseteq H'-L'$, and $L+\{t_1,t_5\}\subseteq L'$. By the choice of $(H,L)$, $H'$ has no $V(H'\cap L')$-good wheel. So by
Lemma~\ref{5cutobs}, $(H',V(H'\cap L'))$ is the 9-vertex graph in Figure~\ref{obstructions}.
Since $w$ is the unique neighbor of $w_1$ in $H'-L'$,  $wa\in E(H')$, a contradiction.
$\Box$

\begin{itemize}
\item [(3)] $a\ne b$.
\end{itemize}
For, if $a=b$ then $G_1$ has a 5-separation $(H',L')$ such that
$V(H'\cap L')=\{a, c, t_2,t_3,t_4\}$ is independent in $H'$,
$|V(H'-L')|\ge 5$ (by (2)), and $L+\{t_1,t_5\}\subseteq L'$. So by
Lemma~\ref{5cutobs}, $H'$ has a $V(H'\cap L')$-good wheel. Now
$(H',L')$ contradicts the the choice of $(H,L)$,  as $|V(H'\cap
L')\cap V(G_1\cap G_2)|<|S|$. $\Box$

\medskip

We may assume $(w_1Fw_2-w_2)\cap t_1Dt_2=\emptyset$. For,  suppose not.
If $(w_1Fw_2-\{w_1,w_2\})\cap t_1Dt_2\ne \emptyset$ then by (1), $(ii)$ of Lemma~\ref{extension-4cut} holds. So assume
 $w_1\in V(t_1Dt_2)$. Then $G$ has a
4-separation $(G_1',G_2')$ such that $V(G_1'\cap G_2')=\{a,t_1,t_5,w_1\}$, $b\in V(G_1'-G_2')$,
and $G_2+\{t_i:i\in [5]\}\subseteq G_2'$. Now $V(G_1')\setminus V(G_2')=\{b\}$ as otherwise $(G_1',G_2')$ contradicts the choice of $(G_1,G_2)$.
But then we see that $N_H(t_5)\subseteq \{a,b\}$; so $(i)$ of Lemma~\ref{extension-4cut} holds.

We wish to consider the wheel $W_2$ consisting of those vertices and edges of $H$ cofacial with $w_2$.

\begin{itemize}
  \item [(4)] $w_2$ and $t_1$ are not cofacial in $H$, and $w_2,t_3$ are not cofacial in $H$.
\end{itemize}
First, suppose $w_2$ and $t_3$ are cofacial. Then $G_1$ has a 5-separation $(H',L')$ such that
$V(H'\cap L')=\{a,w_1,w_2,t_3,t_4\}$ is independent in $H'$,
$ww_3w_4w\subseteq H'-L'$, and $L+\{t_1,t_2,t_5\}\subseteq L'$.
By the choice of $(H,L)$, $H'$ has no $V(H'\cap L')$-good wheel. So by
Lemma~\ref{5cutobs}, $(H',V(H'\cap L'))$ is the 9-vertex graph in Figure~\ref{obstructions}. This is impossible, as $w$ is the unique neighbor of $w_1$ in $H'-L'$ and $wa\notin E(H')$.

 Now assume that  $w_2,t_1$ are cofacial. Then $c,w_2$ are cofacial as $c\in V(t_1Dt_2)$. So $G$ has a 4-separation $(G_1',G_2')$
  such that $V(G_1'\cap G_2')=\{b,c,w_2,w\}$, $w_1\in V(G_1'-G_2')$, and $G_2+\{t_i:i\in [5]\}\subseteq G_2'$. By the choice of
  $(G_1,G_2)$, $|V(G_1')|=5$.

Suppose $c=t_1$. Then $G$ has a 4-separation $(G_1'',G_2'')$
  such that $V(G_1''\cap G_2'')=\{a, w_1,t_1,t_5\}$, $b\in V(G_1''-G_2'')$, and $G_2+\{t_2,t_3,t_4\}\subseteq G_2''$. By the choice of
  $(G_1,G_2)$, $|V(G_1'')|=5$; so $N_G(b)=\{a,w_1,t_1,t_5\}$ and,
  hence, $N_H(t_5)=\{a,b\}$ and $(i)$ of Lemma~\ref{extension-4cut}
  holds.

Therefore, we may assume $c\ne t_1$.
  Now consider the 5-separation $(H',L')$ in $G_1$
 such that $V(H'\cap L')=\{a,t_5,t_1,w_2,w\}$ is independent in $H'$,
$bw_1c\subseteq H'-L'$ (by (3)), and $L+\{t_2,t_3,t_4\}\subseteq L'$. By the choice of $(H,L)$, $H'$ has no $V(H'\cap L')$-good wheel. So by
Lemma~\ref{5cutobs}, $(H',V(H'\cap L'))$ is the 8-vertex or 9-vertex graph in Figure~\ref{obstructions}.
This is impossible as $w_1$ is the unique neighbor of $w$ in $H'-L'$ and $w_1a\notin E(H')$.
 $\Box$

\medskip
Suppose $w_2t_2\in E(H)$ or  $w_2$ and $t_2$ are not cofacial. Then $W_2$ is $V(H\cap L)$-good.
By Lemma~\ref{extension-4cut}, we may assume that $(i)$ or $(ii)$ or
$(iii)$ or $(iv)$ of Lemma~\ref{extension-4cut} holds for $W_2$ (with $t_2$ as $s$).
By the separation $(H_1,H_2)$ we see that only $(i)$ of  Lemma~\ref{extension-4cut} can hold for $W_2$. Hence, there exists $t_2',t_2''\in
V(W_2)$ such that $N_H(t_2)=\{t_2',t_2''\}$ and $N_G(w_2)\cap V(t_2'F_2t_2'')=\emptyset$, where $F_2=W_2-w_2$.

We define $t_2'=t_2''=t_2$ when  $w_2t_2\notin E(H)$ and $w_2$ and $t_2$ are cofacial. Then

\begin{itemize}
\item [(5)]  $w_1$ and $t_2'$ are not  cofacial in $H$.
\end{itemize}
For, suppose they are. Then, since  $w_2t_2'\notin E(H)$,
  it follows from (1) that,  to avoid the cut $\{w_1,w_2,t_2''\}$ in $G$,   $w_1$ and $t_2''$ must be cofacial in $G_1$ and $w_1Fw_2=w_1w_2$.
Thus,  $G_1$ has a 5-separation $(H',L')$ such that
$V(H'\cap L')=\{a,w_1,t_2'',t_3,t_4\}$ is independent in $H'$,
$\{w, w_2,w_3,w_4\}\subseteq  V(H'-L')$, and $L+\{t_1,t_5\}\subseteq L'$. By the choice of $(H,L)$, $H'$ has no $V(H'\cap L')$-good wheel. So by
Lemma~\ref{5cutobs}, $(H',V(H'\cap L'))$ is the 9-vertex graph in Figure~\ref{obstructions}. But this is impossible as
$w$ is the unique neighbor of $w_1$ in $H'-L'$ and $wa\notin E(H')$. $\Box$

\medskip

Consider the 5-separation $(H',L')$ in $G_1$ such that
$V(H'\cap L')=\{b,c,t_2',w_2,w\}$ is independent in $H'$,
$w_1\in V(H'-L')$, and $L+\{t_i: i\in [5]\}\subseteq L'$.
By the choice of $(H,L)$, $H'$ has no $V(H'\cap L')$-good wheel. So by
Lemma~\ref{5cutobs}, $(H',V(H'\cap L'))$ is one of the graphs in Figure~\ref{obstructions}.
By (5),  $|V(H')|\ge 7$.

\begin{itemize}
\item [(6)] If  $|V(H')|\ge 8$ then $bFw_2=bw_1w_2$; $H_1$ has a path $Q$ from $b$ to $t_5$ and
internally disjoint from $aFb\cup cP_1t_1$; and  $P_3\cup t_3Dt_4\cup (w_3Fw_4-w_4)$ has a path
$R$ from $w_3$ to $t_4$.
\end{itemize}
Note that $w_1$ is the unique  neighbor of $w$ in $H'-L'$; so if  $|V(H')|\ge 8$ then $bFw_2=bw_1w_2$.
Also note that, by the choice of $\{b,c\}$, $H_1$ has a path $Q$ from $b$ to $t_5$ and
internally disjoint from $aFb\cup cP_1t_1$.

Moreover, $(P_3-t_3)\cup t_3Dt_4\cup (w_3Fw_4-w_4)$ has a path
$R$ from $w_3$ to $t_4$. For, otherwise, $w_4\in
V(t_3Dt_4)$.  Hence $H$ has a separation $(H'',L'')$ such that
$V(H''\cap L'')=\{b,c,t_2,t_3,w_4\}$ is independent in $H''$, $\{w,w_1,w_2,w_3\}\subseteq
V(H''-L'')$, and $L'+\{t_1,t_5\}\subseteq L''$. Since $|V(H''\cap
L'')\cap V(G_1\cap G_2)|<|S|$, we see from the choice of $(H,L)$ that
$H''$ has no $V(H''\cap L'')$-good wheel. Then $(H'',V(H''\cap L'')$
is the 9-vertex graph in Figure~\ref{obstructions}. However, this is
not possible, as $w_1$ is the unique neighbor of $a$ in $H''-L''$ and
$w_1w_4\notin E(G)$. $\Box$

\begin{itemize}
\item [(7)] We may assume $|V(H')|=7$.
\end{itemize}
First, suppose $|V(H')|=9$. Then $H'-\{w,w_1,c,t_2'\}$ has a
path $bv_1v_2v_3w_2$ such that $v_i\in N_G(w_1)$ for $i\in [3]$, $v_1,v_2\in N_G(c)$, and $v_2,v_3\in N_G(t_2')$.
Since $t_3\notin S$ and because of $Q$ and $R$, we see that $W_1$, the wheel consisting of vertices and edges of $H$
cofacial with $w_1$, is $(V(H\cap L),S)$-extendable.

Now suppose $|V(H')|=8$. Then $H'-\{w,w_1,c,t_2'\}$ has a
path $bv_1w_2$ such that $v_1\in N_G(w_1)$, $H'-L'$ is a path $w_1v_1v_2$, and either $v_2\in N_G(b)\cap N_G(c)$ and
$v_1,v_2\in N_G(t_2')$, or $v_2\in N_G(w_2)\cap N_G(t_2')$ and
$v_1,v_2\in N_G(c)$. Again, since $t_3\notin S$ and because of $Q$ and
$R$, we see that $W_1$ is $(V(H\cap L),S)$-extendable. $\Box$

\medskip

Thus, let $V(H'-L')=\{w_1,v\}$.  Suppose $v\notin V(w_1P_1c)$. Then
$N_G(v)=\{c,t_2',w_2,w_1\}$. Since $G$ is 4-connected, it follows from the choice of $\{b,c\}$ that
$H_1$ has a path $Q'$ from $b$ to $t_5$ internally disjoint from $aFb\cup W_1\cup cP_1t_1$.
Now, since $t_3\notin S$ and because of $Q'$ and $R$, we see that  $W_1$ is $(V(H\cap L),S)$-extendable.

Hence, we may assume $v\in V(w_1P_1c)$. Then $vw_2\in E(H)$, since $w_1, t_2'$ are not cofacial. Note that  $t_2'v\in E(H)$.
If $bv\in E(H)$ then, since $t_2\notin S$ and because of $Q,R$, we see that $W_1$ is $(V(H\cap L),S)$-extendable. So $bv\notin E(H)$. If $ct_2'\in E(H)$
then let $W_v$ denote the wheel consisting of vertices and edges of $H$ cofacial with $v$; then
by the choice of $\{b,c\}$,
$H_1$ has a path $Q'$ from $b$ to $t_5$ internally disjoint from
$aFb\cup W_v\cup cP_1t_1$, and, hence, since $t_3\notin S$ and because
of $R$,
$W_v$  is $(V(H\cap L),S)$-extendable.
So assume $ct_2'\notin  E(H)$.

We may assume  $c\ne t_1$. For, if $c=t_1$ then $G$ has a 4-separation $(G_1',G_2')$ such that $V(G_1'\cap G_2')=\{a,w_1,c,t_5\}$, $b\in V(G_1'-G_2')$, and
$G_2+\{t_2,t_3,t_4\}\subseteq G_2'$. By the choice of $(G_1,G_2)$, we see that $|V(G_1')|=5$ and $N_G(b)=\{a,t_5,t_1,w_1\}$, which
implies that $(i)$ of Lemma~\ref{extension-4cut} holds for $W$.

Then $G_1$ has a 5-separation $(H'',L'')$ such that $V(H''\cap L'')=\{a,t_5,t_1,v,w_1\}$ is independent in $H''$, $\{b,c\}\subseteq
V(H''-L'')$, and $L+\{t_2,t_3,t_4\}\subseteq L''$.
By the choice of $(H,L)$, $H''$ has no $V(H''\cap L'')$-good wheel. So by
Lemma~\ref{5cutobs}, $(H'',V(H''\cap L''))$ is one of the graphs in Figure~\ref{obstructions}.
Since $b$ is the unique neighbor of $w_1$ in $H''-L''$ and $c$ is the unique neighbor of $v$ in $H''-L''$,
$|V(H'')|=7$. If $t_1b\in E(H)$ then $N_H(t_5)\subseteq \{a,b\}$ and $(i)$ of Lemma~\ref{extension-4cut} holds. So assume $t_1b\notin E(H)$. Then
 $N_G(b)=\{a,t_5,c,w_1\}$ and $N_G(c)=\{b,t_1, t_5,v\}$. Thus, $((H-t_1)-t_5c, L\cup t_5ct_1)$
is a 5-separation in $G_1$ that contradicts the choice of $(H,L)$.
\end{proof}

\medskip

We further eliminate possibilities $(i)$ and $(iii)$ of Lemma~\ref{extension-4cut}.

\begin{lem}\label{extension-4cut2}
With the same assumptions of Lemma~\ref{extension-4cut}, $H$ has a $(V(H\cap L),S)$-extendable wheel, or
$G_1$ has a $V(G_1\cap G_2)$-extendable wheel, or  $(ii)$ of Lemma~\ref{extension-4cut} holds for any $w\in V(H-L)$ and
for any $V(H\cap L)$-good wheel $W$ in $H$ with center $w$.
\end{lem}
\begin{proof}
 By Lemma~\ref{extension-4cut1},  we may assume that $(i)$ or $(iii)$ of  Lemma~\ref{extension-4cut} holds for
   some $V(H\cap L)$-good wheel $W$. Let $w$ be the center of $W$, and let $F=W-w$. Let  $V(H\cap L)=\{t_1, t_2, t_3, t_4,
   t_5\}$. We may assume that $G_1$ is drawn in a closed disc in the plane with no
   edge crossings such that the vertices in $V(G_1\cap G_2)$ occur on
   the boundary of that disc. Further, we may assume that  $(H,t_1,t_2,t_3,t_4,t_5)$ is planar.

 By Lemma~\ref{extension-5cut}, $W$  is $V(H\cap L)$-extendable. So let $P_1,P_2,P_3,P_4$ be paths in $H$ from
    $w$ to $t_1, t_2, t_3, t_4$, respectively, such that $V(P_i\cap P_j)=\{w\}$ for
    distinct $i,j\in [4]$, and $|V(P_i)\cap
    V(W)|=2$ for $i\in [4]$. Let $V(P_i)\cap V(F)=\{w_i\}$ for $i\in [4]$.

 Since $(i)$ or $(iii)$ of  Lemma~\ref{extension-4cut} holds for $W$, we may assume that there exist
 $a,b\in V(w_4Fw_1)$   and separation $(H_1,H_2)$ in $H$, such that $w_4,a,b,w_1$
  occur on $F$ in clockwise order,  $N_G(w)\cap V(aFb)=\emptyset$,
 $V(H_1\cap   H_2)=\{a,b,t_1,t_4\}$, $aFb+t_5\subseteq H_1$, and
   $bFa+\{w,t_2,t_3\}\subseteq H_2$. Moreover, $t_5\in S$; and   $t_1,t_4\in S$, or
   $H_1$ consists of the triangle  $abt_5a$ (or the edge $t_5a=t_5b$) and two isolated vertices
   $t_1$ and $t_4$.

We choose $a,b$ so that $aFb$ is minimal. We further choose $P_1,P_2,P_3,P_4$ to minimize $w_4Fw_1$ and then
$w_3Fw_2$. By the same argument in the proof of
Lemma~\ref{extension-4cut1}, we have
\begin{itemize}
\item [(1)]   $N_G(w)\cap V(w_3Fw_2)=\{w_1,w_2,w_3,w_4\}$.
\end{itemize}

Note that $w_2\ne t_2$ or $w_3\ne t_3$. Since, otherwise, $w_2Fw_3=t_2t_3\in E(H)$ (as $G$ is 4-connected), contradicting the fact that $V(H\cap L)$ is
independent in $H$. We claim that

\begin{itemize}
\item [(2)]  $w_1, t_2$ are not cofacial in $H$ and that
 $w_4, t_3$ are not cofacial in $H$.
\end{itemize}
For, suppose otherwise and assume by symmetry that $w_1$ and $
t_2$ are cofacial in $H$.
Then $w_4\ne t_4$, to avoid the 4-separation $(G_1',G_2')$ in $G$
such that $V(G_1'\cap G_2')=\{t_2,t_3,w_4, w_1\}$, $\{w,w_2\}\subseteq
V(G_1'-G_2')$ or $\{w,w_3\}\subseteq
V(G_1'-G_2')$, and $G_2+t_5\subseteq G_2'$.

Suppose $w_3=t_3$. Then $w_2\ne t_2$ and
$G$ has a 4-separation $(G_1',G_2')$
such that $V(G_1'\cap G_2')=\{w_1,t_2, w_3,w_4\}$, $\{w,w_2\}\subseteq
V(G_1'-G_2')$, and $G_2+t_5\subseteq G_2'$. Now $(G_1', G_2')$
contradicts the choice of $(G_1,G_2)$.

So $w_3\ne t_3$.
Then $G_1$ has  a 5-separation $(H', L')$  such that
$V(H'\cap L')=\{a, w_1, t_2, t_3, t_4\}$ is independent in $H'$, $ww_3w_4w\subseteq
H'- L'$, and $L+\{b, t_1, t_5\}\subseteq L'$. By the choice of $(H, L)$,
$H'$ does not contain any $V(H'\cap L')$-good wheel. So by Lemma
\ref{5cutobs}, $(H',V(H'\cap L'))$ is the 9-vertex graph in
Figure~\ref{obstructions}.  This is impossible  because one of the following holds:  $ w_4$ is the unique neighbor of $a$ in $H'-L'$ but
$w_1w_4\notin E(H')$, or  $w$ is the unique neighbor of $w_1$ in $H'-L'$ and
$aw\notin E(H')$. $\Box$

\medskip
Thus, $w_1\ne t_1$ (as $t_1,t_2$ are cofacial in $H$), $w_2\ne t_2$ (as $w_1,w_2$ are cofacial in $H$),
$w_3\ne t_3$ (as $w_3,w_4$ are cofacial in $H$),  and  $w_4\ne t_4$
(as $t_4,t_3$ are cofacial in $H$). Moreover,

\begin{itemize}
\item [(3)] $a\ne b$.
\end{itemize}
For, suppose $a=b$. Then $G_1$ has a 5-separation $(H',L')$ such that $V(H'\cap L')=\{b, t_1,t_2,t_3,t_4\}$ is independent in $H'$, $\{w,w_1,w_2,w_3,w_4\}\subseteq V(H'-L')$,
$L+t_5\subseteq L'$.  Hence, by Lemma~\ref{5cutobs}, $H'$ has a
$V(H'\cap L')$-good wheel. Now $(H',L')$ contradicts the choice of
$(H,L)$. $\Box$

\begin{itemize}
\item [(4)] $t_5$ is not cofacial in $H$ with $w_1$ or $w_4$.
\end{itemize}
 For, otherwise, assume by symmetry that $w_1$ and $t_5$ are cofacial. Then $G$ has a
4-separation $(G_1',G_2')$ such that $V(G_1'\cap G_2')=\{a,t_4,t_5,w_1\}$, $b\in V(G_1'-G_2')$,  and $w_1Fa\cup G_2\subseteq G_2'$. Hence,
by the choice of $(G_1,G_2)$, $|V(G_1')|=5$ and $N_G(b)=\{a,t_4,t_5,w_1\}$.
Therefore, we could have chosen $a=b$, contradicting (3) and the minimality of
$aFb$. $\Box$

\begin{itemize}
\item[(5)]  $w_2,t_3$ are not cofacial in $H$ and that $w_3,t_2$ are
  not cofacial in $H$.
\end{itemize}
For, suppose this is false and assume by symmetry that $w_2$ and $t_3$ are cofacial in $H$.
Then $G_1$ has a 5-separation $(H',L')$ such that $V(H'\cap
L')=\{a,w_1,w_2,t_3,t_4\}$ is independent in $H'$, $ww_3w_4w\subseteq H'-L'$, and $L+\{t_1,t_2,t_5\}\subseteq L'$.
By the choice of $(H,L)$, $H'$ does not contain any $V(H'\cap
L')$-good wheel. So $(H',V(H'\cap L'))$ must be the
9-vertex graph in Figure~\ref{obstructions}.
However, this is not possible, because one of the following holds:
$w_4$ is the unique neighbor of $a$ in $H'-L'$ but $w_4w_1\notin E(H')$, or  $w$ is the unique neighbor of $w_1$ in $H'-L'$ and
$aw\notin E(H')$. $\Box$

\medskip

Suppose $\{t_2,t_3\}\subseteq S$. Then  $G$ has a separation
$(G_1',G_2')$ such that $V(G_1'\cap G_2')=S\cup \{t_4\}$ or $V(G_1'\cap G_2')=S\cup \{t_1\}$, $H\subseteq
G_1'$, and $G_2+t_1\subseteq G_2'$ or $G_2+t_4\subseteq
G_2'$. However, $(G_1',G_2')$ contradicts the choice of
$(G_1,G_2)$. Thus, we may assume

\begin{itemize}
\item [(6)] $t_2\notin S$.
\end{itemize}

We will consider wheels $W_i$ (for $i\in [2]$) consisting of the vertices and edges of $H$ that are cofacial with $w_i$.

\begin{itemize}
\item [(7)] $H_2$ has disjoint paths $P,Q$ from $w_2,w_3$ to $t_3,t_4$, respectively, and
 internally disjoint from $w_4Fa\cup bFw_2\cup P_1$;  and $H_1$ has a path $R$ from $b$ to $t_5$ and
internally disjoint from $W_1\cup aFb+t_4$.
\end{itemize}
First, suppose $P,Q$ do  not exist. Then there exist $v_3\in V(P_3)\setminus \{t_3,w_3\}$ and separation $(G_1',G_2')$ in $G$ such that
$V(G_1'\cap G_2')=\{w_1,w_2,v_3,w_4\}$, $\{w,w_3\}\subseteq  V(G_1'-G_2')$, and
$G_2\cup L\subseteq G_2'$. Now $(G_1',G_2')$ contradicts the choice of $(G_1,G_2)$.

Now assume that the  path $R$ does not exist. Then $H_1$ has a 2-cut $\{p,q\}$
 separating $b$ from $t_5$ such that $p\in V(W_1-b)\cup \{t_1\}$ and
 $q\in V(aFb-b)\cup \{t_4\}$.

If $p=t_1$ then $q=t_4$ by the minimality of $aFb$. So $G$ has a
separation $(G_1',G_2')$ such that $V(G_1'\cap
G_2')=\{t_1,t_2,t_3,t_4\}$, $G_1'\subseteq G_1-t_5$, and
$G_2+t_5\subseteq G_2'$. Now $(G_1',G_2')$ contradicts the choice of
$(G_1,G_2)$.

Hence, $p\in V(W_1-b)$.   Then $q\notin V(aFb-b)$ to avoid the 3-cut $\{p,q,w_1\}$ in $G$. So $q=t_4$. Now
$G_1$ has a 5-separation $(H',L')$ such that $V(H'\cap L')=\{p,t_1,t_2,t_3,t_4\}$ is independent
in $H'$, $\{w, w_1, w_2, w_3, w_4\}\subseteq  V(H'-L')$, and $L+t_5\subseteq L'$. By
Lemma~\ref{5cutobs}, $H'$ has a $V(H'\cap L')$-good wheel. So
$(H',L')$ contradicts the choice of $(H,L)$. $\Box$

 \begin{itemize}
\item [(8)] We may assume that $w_2$ and $t_1$ are not cofacial in $H$.
\end{itemize}
For, otherwise, $G$ has a 4-separation $(G_1',G_2')$ such that
$V(G_1'\cap G_2')=\{b,t_1,w_2,w\}$, $w_1\in V(G_1'-G_2')$, and
$G_2+\{t_2,t_3,t_4,t_5\}\subseteq G_2'$.
Hence, $V(G_1'-G_2')=\{w_1\}$ by the choice of $(G_1,G_2)$. Since $w_1$ and $t_2$ are not cofacial in $H$, $w_2t_1\in E(H)$.
Now the paths $w_1t_1, R\cup w_1b,P\cup w_1w_2, Q\cup w_1ww_3$ show that
$W_1$ is $(V(H\cap L),S)$-extendable. $\Box$

\medskip

Then $w_2t_2\notin E(H)$ and $w_2$ and $t_2$ are cofacial. For,
otherwise, $W_2$ is a $V(H\cap L)$-good wheel in $H$. So by
Lemma~\ref{extension-4cut1},  $(i)$ or $(ii)$ or $(iii)$  of
Lemma~\ref{extension-4cut} occurs for $W_2$. Since  $w\in
W_2$ and $W_2$ is disjoint from $w_3Fw_1-\{w_1,w_3\}$, $(i)$ of
Lemma~\ref{extension-4cut} occurs with  $t_2\in S$, contradicting (6).

Hence, $G_1$ has a 5-separation $(H',L')$ such that $V(H'\cap L')=\{b,t_1,t_2,w_2,w\}$ is independent
in $H'$, $w_1\in V(H'-L')$, and $L+\{t_3,t_4,t_5\}\subseteq L'$. By the choice of $(H,L)$,
$H'$ does not contain any $V(H'\cap L')$-good wheel. Hence by
Lemma~\ref{5cutobs}, $(H', V(H'\cap L'))$ is one of the graphs in Figure~\ref{obstructions}.
Note that $|V(H')|\ge 7$ by (2). We may assume that

\begin{itemize}
\item [(9)] 
$|V(H')|=7$ and  $N_G(t)=\{t_1,t_2,w_2,w_1\}$ with $t\in V(H'-L')\setminus \{w\}$.
\end{itemize}
First, we may assume $|V(H')|=7$. For, suppose $|V(H')|\ge 8$. Then, since
$w_1$ is the only neighbor of $w$ in $H'-L'$,  we see, by checking the 8-vertex and 9-vertex graph in Figure~\ref{obstructions},
that $bFw_2=bw_1w_2$, $W_1$ is defined, and $P_1$ can
be chosen so that $W_1-w_1$
intersects $P_1-w$ just once. So  the paths $w_1P_1t_1, R\cup w_1b,P\cup  w_1w_2, Q\cup w_1ww_3$
show that $W_1$ is $(V(H\cap L),S)$-extendable.

Now let $t\in V(H'-L')\setminus \{w\}$. We may assume $N_G(t)=\{t_1,t_2,w_2,w_1\}$. This is clear if $P_1=ww_1t_1$. So assume
 $P_1=ww_1tt_1$. Then $tw_2\in E(H')$ by (2). If $tb\in
E(H')$ then $W_1$ is  a $V(H\cap L)$-good wheel, and
 $P_1-w, R\cup w_1b, P\cup  w_1w_2, Q\cup w_1ww_3$
show that $W_1$ is $(V(H\cap L),S)$-extendable. So assume $tb\notin
E(H')$. Hence, $N_G(t)=\{t_1,t_2,w_2,w_1\}$.  $\Box$

\begin{itemize}
\item [(10)]  $t_1b\in E(H)$.
\end{itemize}
 For, suppose   $t_1b\notin E(H)$. Consider the 5-separation $(H'',L'')$ in $G_1$
such that $V(H''\cap L'')=\{a,t_4,t_5, t_1,w_1\}$ is independent
in $H''$, $b\in V(H''-L'')$, $|V(H''-L'')|\ge 2$ (because of $R$ and $W_1$), and $L+\{t_2,t_3\}\subseteq L''$. By the choice of $(H,L)$ and by
Lemma~\ref{5cutobs}, $(H'', V(H''\cap L''))$ is one of the graphs in Figure~\ref{obstructions}.
Note that $b$ is the only neighbor of $w_1$ in $V(H''-L'')$. Since $t_1b\notin E(H)$, $|V(H'')|=7$. Because of $R$ and $W_1$, we see that
$R=bt_5$ and, hence, $\{b,t_1,t_5\}$ is a 3-cut in $G$,  a
contradiction. $\Box$

\medskip
Suppose $P_1=ww_1t_1$. If  $w_1Fw_2=w_1w_2$ then $w_1t_1, R\cup w_1b, P\cup  w_1w_2, Q\cup w_1ww_3$
show that $W_1$ is $(V(H\cap L),S)$-extendable. So assume $w_1Fw_2=w_1tw_2$. If there are
disjoint paths $P',Q'$ in $H$ from $t_2,w_3$ to $t_3,t_4$,
respectively, and internally disjoint from $w_4Fa\cup bFw_2\cup P_1$, then  $w_1t_1, R\cup w_1b, P'\cup  w_1tt_2, Q'\cup w_1ww_3$
show that $W_1$ is $(V(H\cap L),S)$-extendable. Hence, we may assume that  $P',Q'$ do  not exist.
Then there exist $v_3\in V(P_3)\setminus \{t_3,w_3\}$ and separation $(H'',L'')$ in $H$ such that
$V(H''\cap L'')=\{w_4, w_1, t, t_2, v_3\}$ is independent in $H''$, $ww_2w_3w\subseteq  H''-L''$, and
$L'+\{t_1,t_5\}\subseteq L''$.  By the choice of $(H,L)$,
$H''$ does not contain any $V(H''\cap L'')$-good wheel. Hence by
Lemma~\ref{5cutobs}, $(H'', V(H''\cap L''))$ is one of the graphs in Figure~\ref{obstructions}.
But this is not possible as $w$ is the unique neighbor of $w_1$ in $H''-L''$ and $wt\notin E(H'')$.

Therefore, $P_1=ww_1tt_1$.
Let $G':=G-\{t,w_1\}+t_1w$, which does not contain a $K_5$-subdivision
as $t_1w$ can be replaced by
$t_1tw_1w$. So $G'$ admits a 4-coloring, say $\sigma$. We now have a
contradiction by extending $\sigma$
to a 4-coloring of $G$ as follows: If $\sigma(t_1)= \sigma(w_2)$ then
greedily color $w_1,t$ in order; if $\sigma(t_1)\ne \sigma(w_2)$
then assign $\sigma(t_1)$ to $w_1$ and greedily color $t$.
\end{proof}

\section{Proof of  Theorem~\ref{main}}

Suppose that $G$ is a Haj\'{o}s graph and that $G$ has a 4-separation $(G_1,G_2)$ such that $(G_1,V(G_1\cap G_2))$ is planar and $|V(G_1)|\ge 6$,
and choose such $(G_1,G_2)$ that $G_1$ is minimal. 
Further, we assume that
$G_1$ is drawn in a closed disc in the plane with no edge crossings such that $V(G_1\cap G_2)$ is contained in the boundary of that disc.

By Lemma~\ref{5cutobs}, $G_1$ has a $V(G_1\cap G_2)$-good
wheel. Moreover, by Lemma~\ref{nogoodwheel}, any $V(G_1\cap G_2)$-good
wheel in $G_1$ is not $V(G_1\cap G_2)$-extendable. Hence,  by Lemma \ref{extension-5cut},
there exists a 5-separation $(H, L)$ in $G_1$ such that $V(H\cap L)$ is independent in $H$,  $V(G_1\cap
G_2)\subseteq V(L)$, $V(G_1\cap G_2)\not\subseteq V(H\cap L)$, and $H$
has a $V(H\cap L)$-good wheel. Let $S=V(H\cap L)\cap V(G_1\cap G_2)$.
We further choose $(H, L)$ such that
\begin{itemize}
\item [(1)] $|S|$ is minimum and, subject to this, $H$ is minimal.
\end{itemize}
Then by Lemma~\ref{extension-5cut},
\begin{itemize}
\item [(2)] any $V(H\cap L)$-good wheel in $H$ is $V(H\cap L)$-extendable.
\end{itemize}
 Let $V(H\cap L)=\{t_1, t_2, t_3, t_4, t_5\}$ such that  $(H,t_1,t_2,t_3,t_4,t_5)$ is planar.
Note that
\begin{itemize}
\item [(3)] the vertices in $S$ must occur consecutively in the cyclic ordering $t_1,t_2,t_3,t_4,t_5$.
\end{itemize}
For, suppose not. Then, without loss of generality,
assume that $t_1,t_3\in S$ but $t_2,t_5\notin S$. Let $V(G_1\cap G_2)=\{t_1,t_3,x,y\}$.

If $(G_1,t_1,x,t_3,y)$ is planar then there exists a 4-separation $(G_1',G_2')$ in $G$ such that
$V(G_1'\cap G_2')=\{t_1,t_2,t_3,y\}$, $H\subseteq G_1'$, $x\notin V(G_1')$, and
$G_2 \subseteq G_2'$; which  contradicts the choice of $(G_1,G_2)$. Similarly, if $(G_1,t_1,y,t_3,x)$ is planar we obtain a contradiction.

If $(G_1,t_1,x,y,t_3)$ or $(G_1,t_1,y,x,t_3)$  is planar then  $\{t_1,t_2,t_3\}$ would be a 3-cut in $G$.

So assume $(G,t_1,t_3,x,y)$ is planar (by renaming $x,y$ if necessary). Then
$G$ has a 4-separation $(G_1',G_2')$  such that $V(G_1'\cap G_2')=\{t_1,t_3,t_4,t_5\}$, $H\subseteq G_1'$, $\{x,y\}\not\subseteq V(G_1')$, and
$G_2 \subseteq G_2'$, which contradicts the choice of $(G_1,G_2)$. $\Box$

\medskip

 We claim that

\begin{itemize}
\item [(4)] no $V(H\cap L)$-good wheel in $H$ is  $(V(H\cap L),S)$-extendable.
\end{itemize}
For, suppose $W$ is a $V(H\cap L)$-good wheel in $H$ that is also $(V(H\cap L),S)$-extendable. Let $w$ be the center of $W$ and
assume that  $H$ has four  paths $P_1,P_2,P_3,P_4$ from
    $w$ to $t_1, t_2, t_3, t_4$, respectively, such that $V(P_i\cap P_j)=\{w\}$ for
    distinct $i,j\in [4]$, $|V(P_i)\cap
    V(W)|=2$ for $i\in [4]$, and $S\subseteq \{t_1, t_2,t_3, t_4\}$.

Let $k=4-|S|$. Since $W$ is not
$V(G_1\cap G_2)$-extendable, $L-(S\cup \{t_5\})$ does not contain
$k$ disjoint paths from $\{t_i:
 i\in [4]\}\setminus S$ to $V(G_1\cap G_2)\setminus S$.
Thus, $L-(S\cup \{t_5\})$ has a cut $T$ of size at most $k-1$ separating
 $\{t_i:
 i\in [4]\}\setminus S$ from $V(G_1\cap G_2)\setminus S$. Hence $T\cup S\cup \{t_5\}$ is a cut in
 $G$, and  $|T\cup
 S\cup \{t_5\}|=4$ since $G$ is 4-connected. Thus, $G$ has a
 4-separation $(G_1', G_2')$  such that $V(G_1'\cap G_2')=T\cup S\cup
 \{t_5\}$, $H\subseteq G_1'$, $G_1'$ is a proper subgraph of $G_1$, and $G_2\subseteq
 G_2'$. Note that $|V(G_1')|\geq 6$ because $W\subseteq H\subseteq G_1'$ and $V(H\cap L)$ is independent in
 $H$; so $(G_1',G_2')$ contradicts the choice of $(G_1, G_2)$. $\Box$

\medskip

Thus, by (4) and Lemma~\ref{extension-4cut2},
\begin{itemize}
\item [(5)] for any $V(H\cap L)$-good wheel $W$ in $H$ with center $w$,
$(ii)$ of Lemma~\ref{extension-4cut} holds.
\end{itemize}

By (2) (and without loss of generality), let  $P_1,P_2,P_3,P_4$ be
paths in $H$ from
    $w$ to $t_1, t_2, t_3, t_4$, respectively, such that $V(P_i\cap P_j)=\{w\}$ for
    distinct $i,j\in [4]$, and $|V(P_i)\cap
    V(W)|=2$ for $i\in [4]$.
Let $F=W-w$ (which is a cycle) and
let $V(P_i)\cap V(F)=\{w_i\}$ for $i\in [4]$. By (4), $t_5\in S$.



By (5),  there exist $s_1,s_2\in S\setminus V(W)$, $a, b\in V(W-w)\setminus N_G(w)$, and a separation $(H_1,H_2)$ in $H$ such that $|V(aFb)\cap N_G(w)|=1$,
$V(H_1\cap H_2)=\{a,b,w\}$, $V(aFb)\cup \{s_1, s_2\}\subseteq V(H_1)$,
and $V(H\cap L)\setminus \{s_1, s_2\}\subseteq V(H_2)$.
Without loss of generality, we may assume that $s_1=t_1, s_2=t_5$, $aFb\subseteq w_4Fw_2$, and $w_1\in V(aFb)$.
We choose $P_i$, $i\in [4]$, to minimize $w_4Fw_2$. Then it is easy to see that
\begin{itemize}
\item [(6)] $N_G(w)\cap V(w_4Fw_2)=\{w_1, w_2, w_4\}$.
\end{itemize}

We claim that

   \begin{itemize}
   \item [(7)]  $w_i\ne t_i$ for $i=2,3,4$.
   \end{itemize}
First, we show  $w_2\ne t_2$ and $w_4\ne t_4$. For, suppose the contrary and, by symmetry, assume $w_2=t_2$. Then $w_3\ne t_3$
as otherwise $w_2Fw_3=w_2w_3$ (since $G$ is 4-connected); so $t_2t_3\in E(H)$,
contradicting the fact that $W\subseteq H$ and $V(H\cap L)$ is independent in $H$. So $w_4\ne t_4$ to avoid the 4-separation $(G_1',G_2')$ with $V(G_1'\cap G_2')=\{w_1,w_2,t_3,w_4\}$,
$\{w,w_3\}\subseteq V(G_1'-G_2')$, and $G_2\subseteq G_2'$. Now $G_1$ has a 5-separation $(H',L')$ such that
$V(H'\cap L')=\{a,w_1,w_2,t_3,t_4\}$ is independent in $H'$,
$ww_3w_4w\subseteq H'-L'$, and $L\cup H_1\subseteq L'$. By the choice of $(H,L)$, $H'$ does not have any $V(H'\cap L')$-good wheel. Hence,
by Lemma~\ref{5cutobs}, $(H',V(H'\cap L'))$ must be the 9-vertex graph in Figure~\ref{obstructions}.
However, this is not possible,  as $w$ is the only neighbor of $w_1$ in $H'-L'$ and $wa\notin E(H)$.

Now suppose $w_3= t_3$.
Then  $G$ has a 4-separation $(G_1',G_2')$ such that $V(G_1'\cap G_2')=\{b,t_2,w_3,w\}$, $w_2\in V(G_1'-G_2')$, and $G_2\cup w_3Fw_1\subseteq G_2'$. Hence, by the
choice of $(G_1,G_2)$, we have $|V(G_1')|=5$. So $N_G(w_2)=\{b,t_2,w_3,w\}$ and $b$ has degree at most 2 in $H_2$.
Similarly, by considering the 4-cut $\{a,t_4,w,w_3\}$, we have $N_G(w_4)=\{a,t_4,w_3,w\}$ and $a$ has degree at most 2 in $H_2$.
 Thus, since $G$ is 4-connected, $a,b$ each have degree at least 2 in $H_1$.

Now consider the 5-separation $(H',L')$ in $G_1$ such that $V(H'\cap L')=\{a, w,b, t_1,t_5\}$ is independent in $H'$, $w_1\in V(H'-L')$,
and $H_2\cup L\subseteq L'$.
By the choice of $(H,L)$, $H'$ does not contain any $V(H'\cap L')$-good wheel. So by Lemma~\ref{5cutobs},
$(H',V(H'\cap L'))$ is one of the
graphs in Figure~\ref{obstructions}. Note that
$w_1$ is the only neighbor of $w$ in $H'-L'$. So $aw_1,bw_1\in E(H)$ as $a$ and $b$ each have degree at least 2 in $H'$.
Note that $H'\subseteq H_1-\{at_5, bt_1\}$.

We may assume that $a$ or $b$ has degree exactly 2 in $H_1$. For, otherwise, by checking
the graphs in Figure~\ref{obstructions}, we see that $|V(H')|=8$ or $|V(H')|=9$, and
$H_1$ contains a wheel $W'$ with center $w'$ such that $N_G(w')=V(W'-w')$ and $|V(W')|\in \{4,5\}$.
If $|V(W')|=4$ then, since $G$ is 4-connected, $G_1-w'$ has four disjoint paths from $V(W')$ to $V(G_1\cap G_2)$; which shows that
$W'$ is $(V(H\cap L),S)$-extendable, contradicting (4). So $|V(W')|=5$. Then, by the choice of $(H,L)$, $H-w$ has 5 disjoint paths from
$V(W')$ to $V(H\cap L)$; which, again, shows that $W'$ is $(V(H\cap L),S)$-extendable, contradicting (4).

Thus,   we may assume by symmetry that $a$ has  degree exactly 2 in $H_1$. Then $a$ has degree 4 in $G$.
Let $\sigma$ be a 4-coloring of $G-\{a,w,w_2,w_4\}$.
If $\sigma(w_1)=\sigma(t_4)$ then by greedily coloring $w_2,w,w_4,a$ in order we obtain a 4-coloring of $G$, a contradiction.
If $\sigma(w_1)=\sigma(t_3)$ then  by greedily coloring $a, w_4,w_2,w$ in order we obtain a 4-coloring of $G$, a contradiction.
So $\sigma(w_1)\notin \{\sigma(t_3),\sigma(t_4)\}$. Then assigning
$\sigma(w_1)$ to $w_4$ and greedily  coloring $w_2,w,a$ in order, we obtain a 4-coloring of $G$, a contradiction. $\Box$

   \begin{itemize}
   \item [(8)]  $w_2,w_4\notin V(D)$, where $D$ denotes the outer walk of $H$.
   \end{itemize}
First, $w_2\notin V(t_2Dt_3)$ and  $w_4\notin V(t_3Dt_4)$.
For, suppose not and assume by symmetry that $w_4\in V(t_3Dt_4)$. Then $G_1$ has a 5-separation $(H',L')$ such that
$V(H'\cap L')=\{b, t_2,t_3,w_4,w_1\}$ is independent in $H'$, $ww_2w_3w\subseteq H'-L'$, $bFw_4+w\subseteq H'$,  and $L+\{t_1,t_5\}\subseteq L'$. By the choice of $(H,L)$, $H'$ has no $V(H'\cap L')$-good wheel. Hence, by  Lemma~\ref{5cutobs}, $(H',V(H'\cap L'))$ is the
9-vertex graph in Figure~\ref{obstructions}.
However, this is impossible since $w$ is the only neighbor of $w_1$ in
$H'$ and $wb\notin E(H')$.

Now suppose (8) fails. Then we may assume by symmetry that
$w_4\in V(D)$. So $w_4\in V(t_4Dt_5)$. Then $w_4Fa=w_4a$ (by (6) and
4-connectedness of $G$) and $G_1$ has a 5-separation $(H',L')$ such that $V(H'\cap L')=\{b,t_1,t_5,w_4,w\}$
is independent in $H'$, $aFw_1\subseteq V(H')\setminus V(L')$,
and $bFw_4+\{t_2,t_3,t_4\}\subseteq L'$. By the choice of $(H,L)$, $H'$
has no $V(H'\cap L')$-good wheel. Hence,
since $w_4$ and $w$ each have exactly one neighbor in $V(H')\setminus V(L')$, it follows from Lemma~\ref{5cutobs} that $V(H')\setminus V(L')=\{a,w_1\}$.
Since $G$ is 4-connected, $N_G(a)=\{t_1,t_5,w_1,w_4\}$ and $N_G(w_1)=\{a,b,t_1,w\}$. However, $G_1$ now has a 5-separation $(H'',L'')$ such that
$H''=(H-t_5)-at_1$ and $L''=L\cup t_1at_5$. Note that $\{w,w_1,w_2,w_3,w_4\}\subseteq V(H'')\setminus V(L'')$; so by Lemma~\ref{5cutobs}, $(H'',L'')$ has a $V(H''\cap L'')$-good
wheel, contradicting the choice of $(H,L)$.  $\Box$

\begin{itemize}
\item [(9)]  $w_2, t_3$ are not cofacial, and $w_4, t_3$ are not cofacial.
\end{itemize}
Otherwise, suppose by symmetry that  $w_2, t_3$ are cofacial. Then $G_1$ has a 5-separation $(H', L')$ such that
$V(H'\cap L')=\{a, w_1, w_2, t_3, t_4\}$ is independent in $H'$,
$ww_3w_4w\subseteq H'-L'$,  and $L+\{b,t_1,t_5\}\subseteq L'$. By the choice of $(H, L)$,
$H'$ does not contain any $V(H'\cap L')$-good wheel. So by Lemma \ref{5cutobs}, $(H',V(H'\cap L'))$ is the 9-vertex graph in Figure~\ref{obstructions}. But this is not possible as
  $w$ is the only neighbor of $w_1$ in $H'-L'$ and $wa\notin E(H')$. $\Box$

\medskip
For $i=2,4$, let $W_i$ denote the wheel consisting of all vertices and edges of $H$
that are cofacial with $w_i$. Then

\begin{itemize}
\item [(10)]   for each $i\in \{2,4\}$, $W_i$ is not a $V(H\cap L)$-good wheel in $H$.
\end{itemize}
For, suppose $W_2$ is a $V(H\cap L)$-good wheel in $H$.
Since $W_2\cap (w_3Fw_1-\{w_1,w_3\})=\emptyset$ and $t_1,t_5\in S$, it follows from
(3) that $(ii)$ of Lemma~\ref{extension-4cut} does not hold for $W_2$, contradicting (5).  $\Box$

\medskip

Thus, by (7)--(10),  $w_2t_2, w_4t_4\notin E(H)$, $w_2$ and $t_2$ are
cofacial in $H$, and $w_4$ and $t_4$ are cofacial in $H$.
Since $G$ is 4-connected, $\{b,t_2, w_2\}$ and $\{a,t_4,w_4\}$ are not cuts in $G$. So by (8), $N_{H_2}(a)=\{t_4,w_4\}$ and
$N_{H_2}(b)=\{t_2,w_2\}$.

 We claim that there exists some $i\in \{2, 3, 4\}$ such that $w_3,
 t_i$ are cofacial and $w_3t_i\notin E(G)$.
For, otherwise, $W_3$ is a $V(H\cap L)$-good in $H$.  Since  $W_3$ is disjoint from $w_4Fw_2-\{w_2,w_4\}$, it follows from (3) that  $(ii)$
of   Lemma~\ref{extension-4cut} does not hold for $W_3$,  contradicting (5).

First, suppose $i\in \{2, 4\}$ and, by symmetry, assume $w_3, t_4$ are cofacial and  $w_3t_4\notin E(G)$.
Then $G$ has a 4-separation $(G_1',G_2')$ such that $V(G_1'\cap G_2')=\{a, t_4,w_3,w\}$,
 $w_4\in V(G_1'-G_2')$,   and $G_2\cup aFw_3\subseteq G_2'$. Since $w_4t_4\notin E(G)$, $|V(G_1'-G_2')|\geq 2$. Hence $(G_1',G_2')$ contradicts
the choice of $(G_1, G_2)$.

Thus,  $w_3, t_3$ are cofacial and  $w_3t_3\notin E(G)$.
By symmetry, we may assume that $P_3-w$ and $w_4$ are on the same side of the face which is incident with both $t_3$ and $w_3$. Now
$G_1$ has a 5-separation $(H',L')$ such that $V(H'\cap L')=\{a, t_4,t_3, w_3,w\}$ is independent in $H'$, $(P_3-w)\cup (P_4-w)\subseteq H'$,
    and $aFw_3+\{t_1,t_2,t_5\}\subseteq L'$. Moreover,   $|V(H'-L')|\geq 3$ since $(P_3-w)\cap (P_4-w)=\emptyset$,
    $w_3P_3t_3\ne w_3t_3$, and $w_4P_4t_4\ne w_4t_4$.
    By the choice of $(H,L)$, $H'$ has no $V(H'\cap L')$-good wheel. Therefore,
  by Lemma~\ref{5cutobs}, $(H',V(H'\cap L'))$ is the 8-vertex graph or 9-vertex graph in Figure~\ref{obstructions}.
  This is impossible, as $w_4$ is the only neighbor of $a$
in $H'-L'$ and $w_4t_4\notin E(H)$, a contradiction.
\qed


\end{document}